\documentclass[12pt]{article}

\usepackage{amsmath}
\usepackage{amsthm}
\usepackage{amssymb}
\usepackage{graphicx}
\usepackage{mathtools}
\usepackage{times}
\usepackage{color}
\usepackage{url}

\usepackage{caption}
\usepackage{subcaption}
\newcommand{\mathsmaller}{}
\newcommand{\parencite}[1]{\cite{#1}}
\newcommand{\textcite}[1]{\cite{#1}}
\newcommand{\leaveout}[1]{}

\DeclareMathOperator{\conv}{conv}

\newcommand{\gesem}{\succeq}

\newcommand{\suchthat}{:}
\newcommand{\tran}{^T}
\newcommand{\set}[1]{\{#1\}}

\newcommand{\Fcal}{{\cal F}}
\newcommand{\Fcalprime}{\Fcal^\prime}
\newcommand{\lam}{\lambda}
\newcommand{\Qhat}{{\hat Q}}
\newcommand{\Rbb}{\mathbb{R}}
\newcommand{\xbar}{{\bar x}}
\newcommand{\ybar}{{\bar y}}
\newcommand{\xhat}{{\hat x}}
\newcommand{\yhat}{{\hat y}}
\newcommand{\xstar}{{x^*}}
\newcommand{\ystar}{{y^*}}

\newcommand{\lx}{{l_x}}
\newcommand{\ly}{{l_y}}
\newcommand{\lz}{{l_z}}
\newcommand{\ux}{{u_x}}
\newcommand{\uy}{{u_y}}
\newcommand{\uz}{{u_z}}

\newtheorem{lemma}{\samb{Proposition}}

\newcommand{\samb}[1]{{#1}}

\title{Convex Hull Representations for\\Bounded Products of Variables}

\author{
Kurt M. Anstreicher\thanks{Department of Business Analytics, Tippie College of Business, University of Iowa, Iowa City, IA 52242. {\tt kurt-anstreicher@uiowa.edu}},\ \
Samuel Burer\thanks{Department of Business Analytics, Tippie College of Business, University of Iowa, Iowa City, IA 52242. {\tt samuel-burer@uiowa.edu}}\ \
and Kyungchan Park\thanks{Department of Business Analytics, Tippie College of Business, University of Iowa, Iowa City, IA 52242. {\tt kyungchan-park@uiowa.edu}}
}

\date{\today}

\begin{document}

\baselineskip22pt

\maketitle

\begin{abstract}
    It is well known that the convex hull of $\set{(x,y,xy)}$, where $(x,y)$ is constrained to lie in a box, is given by the Reformulation-Linearization Technique (RLT) constraints. Belotti {\em et al.\,}(2010) and Miller {\em et al.\,}(2011) showed that
    if there are additional upper and/or lower bounds on the product $z=xy$, then the convex hull can be represented by adding an infinite family of inequalities, requiring a separation algorithm to implement. Nguyen {\em et al.\,}(2018) derived convex hulls with bounds on $z$ for the more general case of $z=x^{\mathsmaller{b_1}} y^{\mathsmaller{b_2}}$, where $b_1\ge 1$, $b_2\ge 1$. We focus on the most important case where $b_1=b_2=1$ and show that the convex hull with \samb{either} an upper bound or lower bound on the product is given by RLT constraints, the bound on $z$ and a single Second-Order Cone (SOC) constraint.  With \samb{both} upper and lower bounds on the product, the convex hull can be represented using no more than three SOC constraints, each applicable o\samb{n} a subset of $(x,y)$ values. In addition to the convex hull characterizations, volumes of the convex hulls with \samb{either} an upper or lower bound on $z$ are calculated and compared to the relaxation that imposes only the RLT constraints. As an application of these volume results, we show how spatial branching can be applied to the product variable so as to minimize the sum of the volumes for the two resulting subproblems.

\end{abstract}

\newpage
\section{Introduction}

Representing the product of two variables is a fundamental problem
in global optimization. This issue arises naturally in the presence
of bilinear terms in the objective and/or constraints, and also
when more complex functions are decomposed in\samb{to} factorable
form by global optimization algorithms such as BARON \cite{BARON}.
It is well known \textcite{Al-Khayyal-Falk} that the convex hull
of $(x,y,xy)$ where $(x,y)$ \samb{lie} in a box is given by the
four Reformulation-Linearization Technique (RLT) constraints
\parencite{RLT,RLTorigin}, also often referred to as the McCormick
inequalities. Linderoth \textcite{Linderoth} derived the convex
hulls of bilinear functions over triangles and showed that they have
Second-Order Cone (SOC) \cite{Convex} representations. \samb{Dey et
al.~\cite{DeySantanaWang} show that the convex hull of $(x,y,xy)$ over
the box intersected with a bilinear equation is SOC representable.} The
convex hull for the complete 5-variable quadratic system that arises
from 2 original variables in a box was considered in \cite{AB} and
\cite{Weismantel}. Explicit functional forms for the convex hull that
apply over a dissection of the box are given in \cite{Weismantel}, while
\textcite{AB} shows that the convex hull can be represented using the
RLT constraints and a PSD condition.

The focus of this paper is to consider the convex hull of $(x,y,xy)$
when $(x,y)$ \samb{lie} in a box and there are explicit upper
and/or lower bounds on the product $xy$. More precisely, we wish to
characterize the convex hull of $$\Fcalprime:=\set{(x,y,z)\suchthat
z=xy,\, \lx\le x\le \ux,\, \ly\le y\le\uy,\, \lz\le z\le \uz}$$ where
\samb{$0 \le (\lx, \ly, \lz) < (\ux,\uy,\uz)$}. \samb{We assume
$\Fcalprime \ne \emptyset$, i.e., that $\lx \ly \le \lz < \uz \le \ux
\uy$.} \samb{When $\lx \ly < \lz$, we say that the lower bound $\lz$ on
$z$ is {\em non-trivial\/} and similarly for the upper bound when $\uz <
\ux \uy$.} By a simple rescaling, we can transform the feasible region to
have $\ux=\uy=1$, and we will make this assumption throughout.
\samb{Note also that if $z=xy$ and ${\ly}>0$ then $x\le {\uz}/{\ly}$,
so we could assume that ${\ux}\le {\uz}/{\ly}$. Then ${\ux}=1$ means
we can assume ${\ly}\le {\uz}$, and similarly ${\lx}\le {\uz}$. In
addition $x\ge {\lz}/{\uy}$, so ${\uy}=1$ implies that we may assume
that ${\lx}\ge {\lz}$ and similarly ${\ly}\ge {\lz}$. Combining these
facts, we could assume that
\begin{equation}\label{eq:bound_assumptions}
 {\lz}\le {\lx}\le {\uz},\quad {\lz}\le
{\ly}\le {\uz}.
\end{equation}
Said differently, if $\lx$ and $\ly$ do not satisfy \eqref{eq:bound_assumptions},
we can adjust them so that they do. However, we do not explicitly assume that
\eqref{eq:bound_assumptions} holds until Section \ref{SEC:Gen}.}

The problem of characterizing $\conv(\Fcalprime)$ has been considered
in several previous works. Bellotti {\em et.al.} \cite{belotti} and
Miller {\em et.al.} \cite{belotti2} show that $\conv(\Fcalprime)$
can be represented by the RLT inequalities, bounds on $z$ and {\em
lifted tangent inequalities}, which we describe in \samb{Section
\ref{SEC:LTI}}. Since the lifted tangent inequalities belong
to an infinite family, they require a separation algorithm to
implement. The convex hull for a generalization of $\Fcalprime$ where
$z=x^{b_1}y^{b_2}$, $b_1\ge1,b_2\ge1$ is considered in \cite{Mohit}.
There are two primary differences between this paper and \cite{Mohit}.
First, because \cite{Mohit} considers a more general problem, both the
analysis required and the representations obtained are substantially
more complex than our results here. In particular, we will show that
with $b_1=b_2=1,$ the convex hull of $\Fcalprime$ can always be
represented using linear inequalities and SOC constraints, although in
some cases the derivations of \samb{the} SOC forms for these constraints
is nontrivial. A second difference is that \cite{Mohit} assumes
$\lx>0$, $\ly>0$. In \cite{Mohit} it is stated that this assumption is without
loss of generality, since by a limiting argument positive lower bounds
could be reduced to zero. This is true, but \cite{Mohit} goes on to
assume that $\lx=\ly=1$, making representations for the important
case of $\lx=0$ and/or $\ly=0$ difficult to extract from the results.
\samb{Another recent, related paper by Santana and Dey \cite{SantanaDey}
shows that $\conv(\Fcalprime)$ is SOC representable using a disjunctive
representation in a lifted space; see section 4 therein. In contrast, we
will show that $\conv(\Fcalprime)$ is SOC representable directly in the
variables $(x,y,z)$.}

In \samb{Section \ref{SEC:lb0}}, we consider the case where $\lx=\ly=0$ and there are \samb{non-trivial} upper and/or lower bounds on the product
variable $z$.  Our methodology for obtaining explicit representations for $\conv(\Fcalprime)$ is based on the lifted
tangent inequalities of \cite{belotti,belotti2}. We do not use the inequalities {\em per se}, but rather show how the process
by which they are constructed can be re-interpreted to generate nonlinear inequalities.  We show that in all cases these
inequalities can be put into the form of SOC constraints, so that $\conv(\Fcalprime)$ is SOC-representable \cite{Convex}.  In the
presence of \samb{both non-trivial} upper and lower bounds on $z$, the representation requires a dissection of the domain of $(x,y)$ values
into three regions, each of which uses a different SOC constraint to obtain the convex hull.  One of the three SOC
constraints is globally valid, and the use of this one constraint together with the RLT constraints and bounds on $z$
\samb{empirically gives} a close approximation of $\conv(\Fcalprime)$.  Finally we compute the
volumes of $\conv(\Fcalprime)$ as given \samb{in the case where there is either a non-trivial upper or non-trivial lower bound
on $z$} using an SOC constraint, the RLT constraints and bound on $z$, and compare these volumes to the volumes of the regions where the SOC constraint is omitted.  This comparison is
similar to the volume computations in \cite{SDPRLT}, where the effect of adding a PSD condition to the RLT constraints was considered. An
interesting application of these computations is to consider the reduction in volume associated with spatial branching
\cite{Spatial} based on the product variable $z$.

In Section \ref{SEC:Gen} we generalize the results of Section
\ref{SEC:lb0} to consider \samb{positive} lower bounds on $(x,y)$,
\samb{specifically} bounds $\lx\ge0$, $\ly\ge0$. We again show
that in all cases $\conv(\Fcalprime)$ is SOC-representable. As in
the case of $\lx=\ly=0$, when there is \samb{a non-trival} upper
or \samb{a non-trivial} lower bound on the product, but not both,
the representation of $\conv(\Fcalprime)$ requires only a single
SOC constraint in addition to the RLT constraints and bound on $z$.
\samb{When there are both non-trivial} lower and upper bounds on
$z$ there are several cases to consider, \samb{again} requiring up
to three SOC constraints, each applicable o\samb{n} a subset of the
domain of $(x,y)$. We close the paper in Section \ref{SEC:Conclusion}
with a summary of the results and some promising directions for future
research.

\section{\samb{Lifted Tangent Inequalities}} \label{SEC:LTI}

The set $\Fcal = \set{(x,y,z)\suchthat z=xy,\, {\lx}\le x\le {\ux},\, {\ly}\le y \le {\uy}}$, \samb{i.e., $\Fcalprime$ with only trivial bounds on $z$}, is not convex, but it
is well known that $\conv(\Fcal)$
 is the linear envelope of four extreme points \parencite{Al-Khayyal-Falk}. This linear envelope can be given by the four RLT constraints \parencite{RLTorigin}:
\begin{subequations}\label{EQ:RLT}
    \begin{align}
        z&\geq {\uy} x+{\ux} y-{\ux} {\uy}, \label{EQ:RLTback}\\
        z&\geq {\ly} x +{\lx} y -{\lx} {\ly}, \label{EQ:RLTbottom}\\
        z&\leq {\uy} x +{\lx} y-{\lx} {\uy}, \label{EQ:RLTx}\\
        z&\leq {\ly} x+{\ux} y-{\ux} {\ly}. \label{EQ:RLTy}
    \end{align}
\end{subequations}
Figure \ref{FG:Target} shows the product $xy$ as a colored surface and the boundary edges for the linear envelope as red lines for the case where ${\lx}={\ly}=0$ and ${\ux}={\uy}=1$.

\begin{figure}[h!]
    \centering
       \includegraphics[height=3.5in]{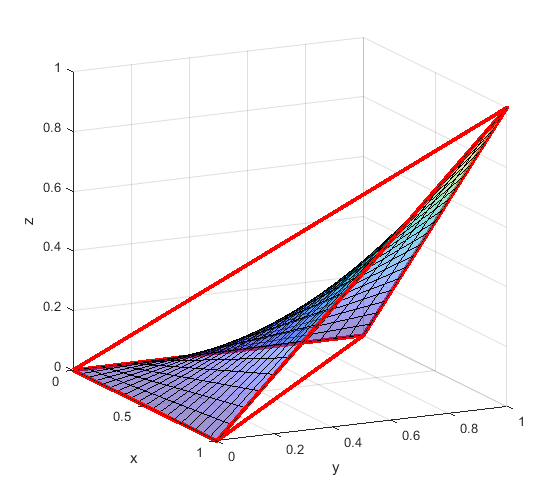}
       \caption{Convex hull with no bounds on $z$}
       \label{FG:Target}
\end{figure}

The focus in this paper is to represent the
convex hull of the set $\Fcalprime$, corresponding to $\Fcal$ with nontrivial upper and lower bounds on the product variable $z$.   \samb{Recall that we} assume ${\ux}={\uy}=1$ throughout.
It is shown in \cite{belotti,belotti2} that the convex hull of $\Fcalprime$ is given by the RLT constraints, bounds on $z$, and {\em lifted tangent inequalities}.
Our technique for deriving convex hull representations \samb{in Sections \ref{SEC:lb0} -- \ref{SEC:Gen}} is based on the construction of these lifted tangent
inequalities, which we now describe.  Assume that $l=({\lx},{\ly},{\lz})$ and $u=(1,1,{\uz})$.
The construction of a lifted tangent inequality based on a point $(\xstar,\ystar,{\lz}) \samb{\in \Fcalprime}$ proceeds as follows.  The inequality tangent to the curve $xy={\lz}$ at $(\xstar,\ystar)$ has the form
\[
\ystar(x-\xstar)+\xstar(y-\ystar)\ge 0.
\]
This inequality is lifted to an inequality in the variables $(x,y,z)$ of the form
\begin{equation}\label{EQ:LTI}
\ystar(x-\xstar)+\xstar(y-\ystar)+a(z-{\lz})\ge 0,
\end{equation}
\samb{with $a < 0$.} The value of $a$ is chosen so that there is a point $(\xbar,\ybar,{\uz}) \samb{\in \Fcalprime}$ such that
the inequality \eqref{EQ:LTI} is tight at $(\xbar,\ybar,{\uz})$, and \eqref{EQ:LTI} is valid for ${\Fcalprime}$. There are two possibilities for such a point:
\begin{itemize}
\item $\xbar=\rho\xstar$, $\ybar=\rho\ystar$, where $\rho=\sqrt{{\uz}/{\lz}}$. In this case the value of $a$ is independent of $(\xstar,\ystar)$; there is an expression for $a$ that depends only on ${\lz}$ and ${\uz}$ \cite{belotti,belotti2}.
\item $(\xbar,\ybar)$ corresponds to one of the endpoints of the curve $xy={\uz}$ for the given bounds
on $x$ and $y$. In the case of ${\uz}=1$, this point is $\xbar=\ybar=1$.
\end{itemize}

The construction of a lifted tangent inequality can alternatively \samb{start} with a point $(\xbar,\ybar,u_z) \in$ $\Fcalprime$. In this case the roles of $(\xstar,\ystar)$ and $(\xbar,\ybar)$ are reversed,
and either $(\xstar,\ystar)=(1/\rho)(\xbar,\ybar)$ or $(\xstar,\ystar)$ is an endpoint of the curve $xy={\lz}$ for
the bounds on $x$ and $y$. If ${\lx}{\ly}={\lz}$ then this point is $({\lx},{\ly},{\lz})$; for example if ${\lx}={\ly}={\lz}=0$, the point is $(0,0,0)$.

In all cases the result of the above process is an inequality that is valid for $\Fcalprime$, and which is tight for a line segment joining
two points $(\xstar,\ystar,{\lz}) \samb{\in \Fcalprime}$ and $(\xbar,\ybar,{\uz}) \samb{\in \Fcalprime}$. Our approach does not use the lifted tangent inequalities themselves but is rather based on
the process for constructing them.  In particular, starting with a point $(x,y)$ with $x\in[{\lx},1]$, $y\in[{\ly},1]$,
${\lz} < xy <{\uz}$, we determine the two points $(\xstar,\ystar,\samb{\lz}) \samb{\in \Fcalprime}$ and $(\xbar,\ybar,\samb{\uz}) \samb{\in \Fcalprime}$ so that the lifted tangent
inequality that is tight at $(x,y,z)$ is tight for the line segment joining $(\xstar,\ystar,{\lz})$ and $(\xbar,\ybar,{\uz})$.
Suppose that $0\le\alpha\le 1$ is such that
$(x,y)=\alpha(\xstar,\ystar)+(1-\alpha) (\xbar,\ybar)$.
Then the constraint $z\le\alpha {\lz}+(1-\alpha){\uz}$ is valid and tight on the line segment between
$(\xstar,\ystar,{\lz})$ and $(\xbar,\ybar,{\uz})$.
If $\alpha$ can be expressed as a function of
$(x,y)$ then the result is a single nonlinear constraint that \samb{is equivalent to} a family of lifted tangent inequalities.  Our goal will be to obtain such a constraint and show that it can always be expressed in the form of an SOC constraint.  An SOC constraint obtained in this manner is certainly
valid over the $\set{(x,y)}$ domain on which it is derived, since it is equivalent to the lifted tangent inequalities on that domain.  In some cases we obtain SOC constraints that are actually globally valid, that is, valid for all $(x,y,z)\in\conv(\Fcalprime)$.

\section{Convex hull representation with $\lx=\ly=0$}\label{SEC:lb0}

In this section
we obtain representations for $\conv(\Fcal^\prime)$ when ${\lx}={\ly}=0$ and ${\ux}={\uy}=1$.
We begin by considering the case where ${\lz}=0$, ${\uz}<1$, and next consider the case where
${\lz}>0$, ${\uz}=1$.  In both of these cases we show that a combination of the RLT constraints, the bound
on $z$, and a single SOC constraint gives the convex hull of $\Fcal^\prime$.  In the case where ${\lz}>0$ and
${\uz}<1$ we show that the convex hull of $\Fcal^\prime$ is representable using three SOC constraints, each applicable on a subset of the domain in $(x,y)$. One of these SOC constraints is globally valid, and the combination
of that single SOC constraint, the RLT constraints and the bounds on $z$ \samb{empirically} gives a close approximation of
$\conv(\Fcalprime)$.

\subsection{\samb{Non-trivial} upper bound on $xy$ with $\lx=\ly=0$} \label{SS:UB}

We first consider the case where $x\in[0,1]$, $y\in[0,1]$ and we impose a \samb{non-trivial} upper bound
on the product $z \le {\uz}$.

\begin{lemma}\label{LEM:UB}
Let $l=(0,0,0)$, $u=(1,1,{\uz})$ where $0<{\uz}<1$. Then $\conv(\Fcalprime)$ is given by the RLT constraints
\eqref{EQ:RLT}, the bound $z\le {\uz}$ and the SOC constraint $z^2\le {\uz}xy$.
\end{lemma}

\begin{proof}
From \cite{belotti, belotti2} the convex hull of $\Fcal^\prime$ is given by the RLT constraints, bounds on $z$ and the lifted tangent inequalities.  In this
case each lifted tangent inequality is obtained by taking a point $(\xbar,\ybar)=(t,{\uz}/t)$ with ${\uz}\le t \le 1$, forming the
tangent equation to $xy={\uz}$ at $(t,{\uz}/t)$, and lifting this tangent equation to form a valid inequality of the form
\begin{equation}\label{EQ:LTEupper}
\frac{{\uz}}{t}x + ty -2z \ge 0.
\end{equation}
The set of points in $\Fcal^\prime$ that satisfy \eqref{EQ:LTEupper} with equality then consists of the line segment
joining the points $(t,{\uz}/t,{\uz})$ and $(0,0,0)$.
The constraint $z^2\le {\uz}xy$ holds with equality for all such points, and therefore implies all of the lifted tangent
inequalities. Moreover, $z^2\le \uz xy$ clearly holds for all $(x,y,z)\in\conv(\Fcalprime)$, since if
$z=xy\le \uz$ then $z^2=(xy)^2 \le\uz xy$.
\end{proof}

\begin{figure}
    \centering
    \includegraphics[height=3in]{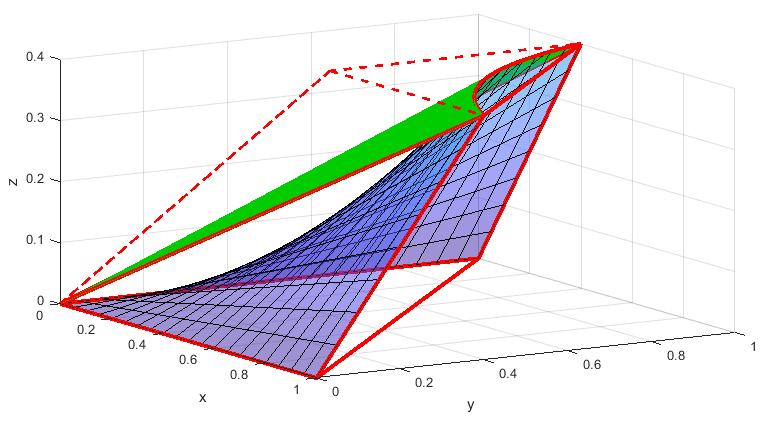}
   \caption{Convex hull with upper bound on $z$ }
   \label{FG:UB}
\end{figure}

In Figure \ref{FG:UB}, we illustrate $\conv(\Fcalprime)$ for the case of $l=(0,0,0)$, $u=(1,1,0.4)$. The green surface illustrates the boundary of the SOC cone, and the red solid lines indicate edges on the boundary of
$\conv(\Fcalprime)$ corresponding to \samb{the} RLT constraints and the upper bound $z\le {\uz}$. The dashed red lines indicate edges of the polyhedron corresponding to the RLT constraints and upper bound on $z$, \samb{highlighting the portion cut away by the SOC constraint}. Note that none of the RLT
constraints are redundant, and although the SOC constraint is globally valid, this constraint does not give a tight
upper bound on $z$ for all feasible $(x,y)$, unlike the case considered in Section \ref{SS:LB}.
In \cite{belotti}, it is noted that if $l=(0,0,{\lz})$ and $u=(+\infty,+\infty, {\uz})$ then
$\conv(\Fcalprime)$ is given by the bounds on $z$ and the SOC constraint $(z+\sqrt{{\lz}{\uz}})^2\le (\sqrt{{\lz}}+\sqrt{{\uz}})^2 xy$\,; when
${\lz}=0$ this is exactly the constraint $z^2\le {\uz}xy$.

\subsection{\samb{Non-trivial} lower bound on $xy$ with $\lx=\ly=0$} \label{SS:LB}

We next consider the case where $x\in[0,1]$, $y\in[0,1]$ and we impose \samb{only} a lower bound
on the product $z=xy \ge {\lz}$. To obtain an SOC representation for $\conv(\Fcalprime)$ we need to
characterize
the lifted tangent inequalities, as in the proof of \samb{Proposition} \ref{LEM:UB}. This is
more complex than for the case of an upper bound $z\le {\uz}$ because the lifted tangent
inequalities are now tight on line segments of the form $\alpha(\xbar,\ybar,{\lz})+(1-\alpha)(1,1,1)$,
where $\xbar\ybar={\lz}$.

\begin{lemma}\label{LEM:LB}
Let $l=(0,0,{\lz})$, $u=(1,1,1)$ where $0<{\lz}<1$. Then $\conv(\Fcalprime)$ is given by the RLT constraint
\eqref{EQ:RLTback}, the bound $z\ge {\lz}$ and the SOC constraint $\sqrt{(\xhat,\yhat) M (\xhat,\yhat)\tran}
\le x+y-2z$ where $\xhat\coloneqq1-x$, $\yhat\coloneqq1-y$, and
\[
    M=\begin{pmatrix}1 & 2l\samb{_z}-1\\2l\samb{_z}-1 & 1\end{pmatrix}\gesem 0.
\]
\end{lemma}

\begin{proof}
Given a point $(x,y)$ with $x>{\lz}$, $y>{\lz}$ and $xy>{\lz}$, a lifted tangent inequality that
is tight at $(x,y,z)$ must have $x=\alpha\xstar+(1-\alpha)$, $y=\alpha\ystar+(1-\alpha)$.
Writing $\xstar$ and $\ystar$ in terms of $x$, $y$ and $\alpha$ and using
$\xstar\ystar={\lz}$ results in a quadratic equation for $\alpha$. Substituting the appropriate root of this
quadratic equation into the
constraint $z\le \alpha {\lz} + (1-\alpha)$ then obtains the equivalent inequality
\[
    z\le \frac{(x+y)-\sqrt{(x-y)^2+4l\samb{_z}(1-x)(1-y)}}{2} .
\]
It is straightforward to verify that
\[
\begin{pmatrix} 1-x \\ 1-y\end{pmatrix}\tran
\begin{pmatrix}1 & 2l\samb{_z}-1\\2l\samb{_z}-1 & 1\end{pmatrix}
\begin{pmatrix} 1-x \\ 1-y\end{pmatrix}
=(x-y)^2+4l\samb{_z}(1-x)(1-y).
\]
The constraint
$\sqrt{(\xhat,\yhat) M (\xhat,\yhat)\tran}\le x+y-2z$
 then implies all of the lifted tangent inequalities, and
$0\le {\lz}\le 1$ implies that $-1\le 1-2{\lz}\le 1$, so $M\gesem 0$.
 Therefore the convex hull of $\Fcalprime$
is given by the RLT constraints, the bound $z\ge {\lz}$ and this one SOC constraint.  However the RLT
constraints \eqref{EQ:RLTbottom}-\eqref{EQ:RLTy} are easily shown to be redundant, even if ${\lx}$
and ${\ly}$ are increased to ${\lz}$ in their definitions.
\end{proof}

\begin{figure}
    \centering
       \includegraphics[height=3.5in]{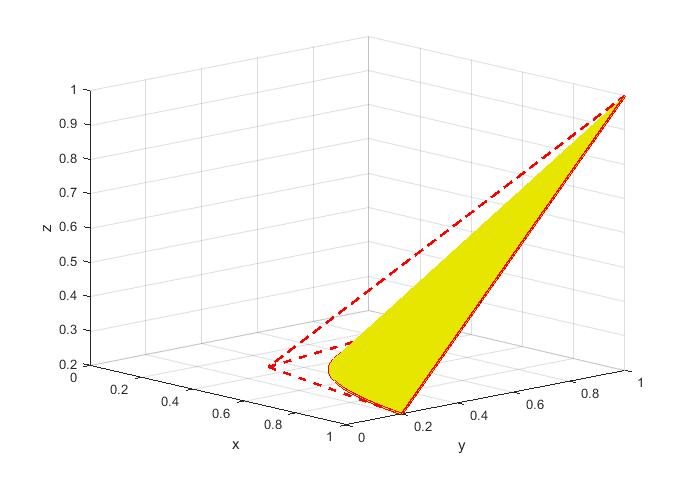}
       \caption{Convex hull with lower bound on $z$}
       \label{FG:LB}
\end{figure}
\par

In Figure \ref{FG:LB}, we illustrate $\conv(\Fcalprime)$ for the case of $l=(0,0,0.2)$, $u=(1,1,1)$. In the figure the dashed lines indicate edges corresponding to the RLT constraints \eqref{EQ:RLTx} -- \eqref{EQ:RLTy}, with ${\lx}$ and
${\ly}$ increased to ${\lz}=0.2$ in the formulas for these constraints, as in \eqref{eq:bound_assumptions}. Note that in this case the SOC constraint gives a tight upper bound on $z$ for all feasible $(x,y)$, unlike the case illustrated in Figure \ref{FG:UB}.

\subsection{\samb{Non-trivial} lower and upper bounds on $xy$ with $\lx=\ly=0$} \label{SS:LBUB}

We now consider the case where both \samb{non-trivial} lower and upper bounds are imposed on the product $z=xy$, so
$0<{\lz} <{\uz}<1$.  The situation becomes more complex than with only an upper or lower bound because now
there are 3 classes of lifted tangent inequalities. In each of the cases below, $(\xstar,\ystar,l_z) \samb{\in \Fcalprime}$.
\begin{enumerate}\baselineskip16pt
\item For \samb{the ``center'' domain} $y \ge {\uz} x$ and $x\ge {\uz} y$, each lifted tangent inequality corresponds to a line segment
    connecting $(\xstar,\ystar,{\lz})$ and $(\xbar,\ybar,{\uz}) \samb{\in \Fcalprime}$, where $\xbar\ybar={\uz}$,
$(\xbar,\ybar)=\rho(\xstar,\ystar)$ and $\rho=\sqrt{{\uz}/{\lz}}$.
\item For \samb{the ``side'' domain} $y \le {\uz} x$, each lifted tangent inequality corresponds to a line segment connecting
$(\xstar,\ystar,{\lz})$ and $(1,{\uz},{\uz})$.
\item For \samb{the ``side'' domain} $x\le {\uz} y$, each lifted tangent inequality corresponds to a line segment connecting $(\xstar,\ystar,{\lz})$
and $({\uz},1,{\uz})$.
\end{enumerate}
\noindent \samb{Figure \ref{FG:threedomains} depicts these three domains in the $xy$-space for the case where $l_z=0.2$, $u_z=0.7$.}

\begin{figure}
    \centering
       \includegraphics[height=3.5in]{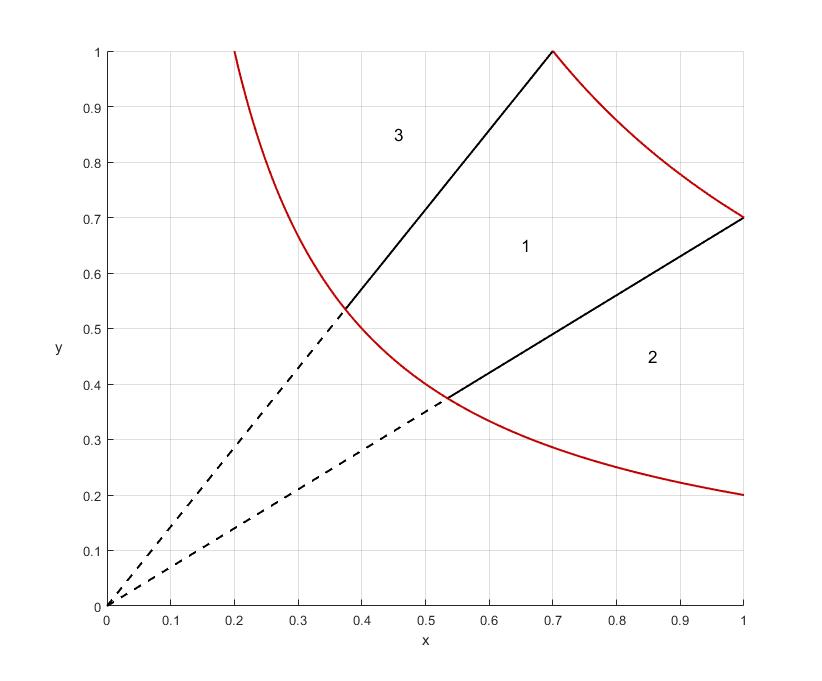}
       \caption{Domains for SOC constraints with lower and upper bounds on $z$}
       \label{FG:threedomains}
\end{figure}

In the lemma below, we show that in this case $\conv(\Fcalprime)$ can be represented using
a single RLT constraint, the bounds on $z$, and 3 different SOC constraints, each applicable on one of the domains
described above.  For convenience in stating the result, \samb{we} define matrices
\begin{equation}\label{EQ:M1M2}
M_1 \samb{:}= \begin{pmatrix} {\uz}^2 & 2{\lz}-{\uz} \\ 2{\lz}-{\uz} & 1 \end{pmatrix}, \quad
M_2 \samb{:}= \begin{pmatrix} 1 & 2{\lz}-{\uz} \\ 2{\lz}-{\uz} & {\uz}^2 \end{pmatrix}.
\end{equation}

\begin{lemma}\label{LEM:LBUB}
Let $l=(0,0,{\lz})$, $u=(1,1,{\uz})$ where $0<{\lz}<{\uz}<1$. Then $\conv(\Fcalprime)$ is given by the RLT constraint
\eqref{EQ:RLTback}, the bounds ${\lz}\le z\le {\uz}$ and three SOC constraints, each applicable in a different region:
\begin{enumerate}
\baselineskip16pt
\item The constraint $(z+\sqrt{{\lz}{\uz}})^2\le (\sqrt{{\lz}}+\sqrt{{\uz}})^2xy$, applicable if $y \ge {\uz} x$ and $x\ge {\uz} y$.
\item  The constraint $\sqrt{(\xhat,\yhat)M_1(\xhat,\yhat)\tran}\le
{\uz}x+y-2z$, where $\xhat\coloneqq 1-x$, $\yhat\coloneqq {\uz}-y$ and $M_1\gesem 0$ is given in \eqref{EQ:M1M2}, applicable if $y\le {\uz} x$.
\item  The constraint $\sqrt{(\xhat,\yhat)M_2(\xhat,\yhat)\tran}\le
x+{\uz}y-2z$, where $\xhat\coloneqq {\uz}-x$, $\yhat\coloneqq 1-y$ and $M_2\gesem 0$ is given in \eqref{EQ:M1M2}, applicable if $x\le {\uz} y$.
\end{enumerate}
\end{lemma}

\begin{proof}
Assume first that $y \ge {\uz} x$ and $x\ge {\uz} y$.  We know that $(x,y)$ is on the line segment connecting
$(\xstar,\ystar)$ and $(\xbar,\ybar)$, from which we conclude that $\xstar=x\sqrt{{\lz}/(xy)}$ and
$\xbar=x\sqrt{{\uz}/(xy)}$.  Then $x=\alpha\xstar+(1-\alpha)\xbar$ implies that
\[
x=\alpha x\sqrt{{\lz}/(xy)} + (1-\alpha)x\sqrt{{\uz}/(xy)},
\]
from which we obtain $\sqrt{xy}=\alpha\sqrt{{\lz}} + (1-\alpha)\sqrt{{\uz}}$, or
\[
\alpha=\frac{\sqrt{{\uz}}-\sqrt{xy}}{\sqrt{{\uz}}-\sqrt{{\lz}}}.
\]
Substituting this value of $\alpha$ into the inequality $z\le \alpha {\lz}+ (1-\alpha){\uz}$ and simplifying, we obtain
the inequality $(z+\sqrt{{\lz}{\uz}})^2\le (\sqrt{{\lz}}+\sqrt{{\uz}})^2xy.$ Therefore, this SOC constraint
implies all of the lifted tangent inequalities if $y \ge {\uz} x$ and $x\ge {\uz} y$.

Next assume that $y\le {\uz} x$.  The situation is now very similar to that encountered in the proof of
\samb{Proposition} \ref{LEM:LB}, except that the lifted tangent inequality is tight on a line segment connecting a point $(\xstar,\ystar,{\lz})$ with $\xstar\ystar={\lz}$  to the point $(1,{\uz},{\uz})$, rather than $(1,1,1)$.  A similar process
to that used in the proof of \samb{Proposition} \ref{LEM:LB} again results in a quadratic equation for $\alpha$ such that
$x=\alpha\xstar+(1-\alpha)$, $y=\alpha\ystar+(1-\alpha){\uz}$, and substituting the appropriate root into the inequality $z\le \alpha {\lz} + (1-\alpha){\uz}$ results in the inequality
\[
 z\le\frac{{\uz}x+y-\sqrt{({\uz}x-y)^2+4{\lz}(1-x)({\uz}-y)}}{2}.
 \]
It is straightforward to verify that
$({\uz}x-y)^2+4{\lz}(1-x)({\uz}-y)= (\xhat,\yhat)M_1(\xhat,\yhat)\tran$,
 where $\xhat\coloneqq 1-x$, $\yhat\coloneqq {\uz}-y$, and $M_1 \gesem 0$ follows from ${\lz}\le {\uz}$.
Therefore, the constraint $\sqrt{(\xhat,\yhat)M_1(\xhat,\yhat)\tran}\le
{\uz}x+y-2z$ implies the lifted tangent inequalities when $y\le {\uz} x$.  The analysis when $x\le {\uz} y$ is
very similar, interchanging the roles of $x$ and $y$.
\end{proof}

Note that if ${\uz}=1$ then $M_1=M_2=M$, where $M\gesem 0$ was given in \samb{Proposition} \ref{LEM:LB}. In this case
we always have either $x\le y$ or $y\le x$, so the ``center'' SOC constraint is not present and the two ``side'' SOC
constraints are identical and equal to the constraint in \samb{Proposition} \ref{LEM:LB}. If ${\lz}=0$, then the SOC constraint that applies when
$y\ge {\uz}x$ and $x\ge {\uz} y$ is identical to the SOC constraint from \samb{Proposition} \ref{LEM:UB}. Moreover, if
$\lz=0$ and $y\le {\uz}x$,
then $(\xhat,\yhat)M_1(\xhat,\yhat)\tran= ({\uz}x-y)^2$, and the SOC constraint
$\sqrt{(\xhat,\yhat)M_1(\xhat,\yhat)\tran}\le {\uz}x+y-2z$ is exactly the RLT constraint $z\le y$.  Similarly for $\lz=0$ and
$x\le {\uz}y$, the SOC constraint $\sqrt{(\xhat,\yhat)M_2(\xhat,\yhat)\tran}\le x+{\uz}y-2z$ becomes the RLT
constraint $z\le x$.

\begin{figure}
    \centering
   \includegraphics[height=3.5in]{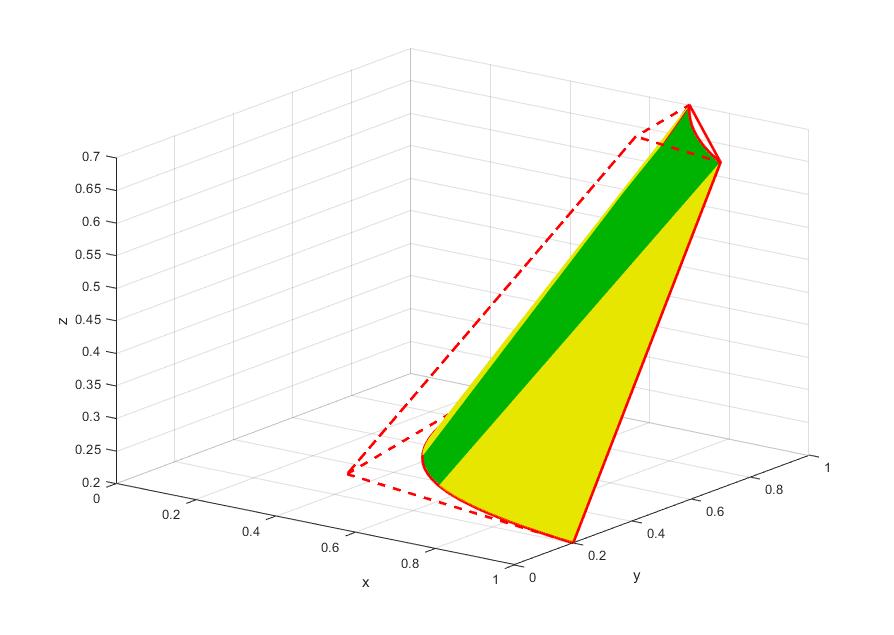}
   \caption{Convex hull with lower and upper bounds on $z$}
   \label{FG:DB}
\end{figure}

In Figure \ref{FG:DB}, we illustrate $\conv(\Fcalprime)$ for $l=(0,0,0.2)$, $u=(1,1,0.7)$.
As in Figure \ref{FG:LB} the dashed red lines indicate edges corresponding to the RLT constraints \eqref{EQ:RLTx} -- \eqref{EQ:RLTy}, with ${\lx}$ and ${\ly}$ increased to ${\lz}=0.2$ in the formulas for these constraints, as in \eqref{eq:bound_assumptions}.

\begin{figure}
    \centering
   \includegraphics[height=3.5in]{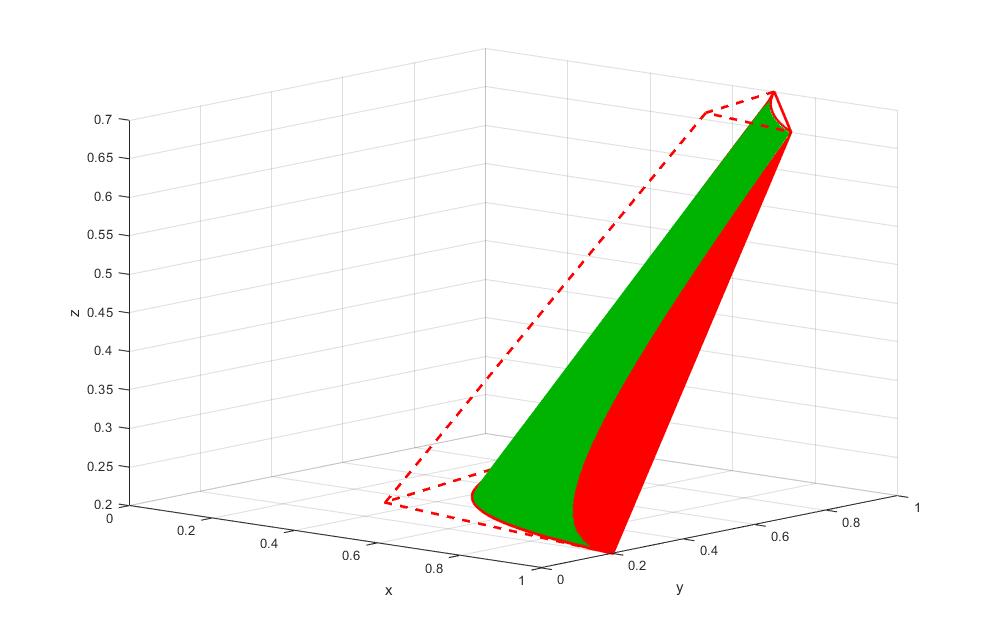}
   \caption{Center cone only with RLT constraints}
   \label{FG:DBonlyC}
\end{figure}

It is easy to show that the ``side'' SOC constraints from \samb{Proposition} \ref{LEM:LBUB}
that are applicable on the domains $y\le {\uz} x$ and $x\le {\uz} y$ are
{\em not} valid outside these domains.  However the ``center'' constraint is valid for all $(x,y,z)\in\conv(\Fcalprime)$.
To see this, note that if ${\lz}\le z=xy \le {\uz}$, then
\begin{eqnarray}
(\sqrt{{\uz}}-\sqrt{xy})(\sqrt{xy}-\sqrt{{\lz}})&\ge& 0\nonumber\\
(\sqrt{{\lz}}+\sqrt{{\uz}})\sqrt{xy} &\ge& xy + \sqrt{{\lz}{\uz}}\nonumber\\
(z+\sqrt{{\lz}{\uz}})^2 &\le& (\sqrt{{\lz}}+\sqrt{{\uz}})^2xy.\label{EQ:LBUBcenter}
\end{eqnarray}
The fact that the center constraint is globally valid means that we can approximate $\conv(\Fcalprime)$ by
using this one SOC constraint together with the RLT constraints \eqref{EQ:RLTx} -- \eqref{EQ:RLTy}, where these
RLT constraints can be tightened by using the values ${\lx}={\ly}={\lz}$ in their definitions, as in \eqref{eq:bound_assumptions}.
We illustrate this approximation in
Figure \ref{FG:DBonlyC} for the case where ${\lz}=0.2$, ${\uz}=0.7$, as in Figure \ref{FG:DB}.
It appears that the use
of this one SOC constraint together with the RLT constraints gives a very close approximation of $\conv(\Fcalprime)$. To
show this more precisely, in Figure \ref{FG:DBslicer} we consider the same case of ${\lz}=0.2$, ${\uz}=0.7$ but show three
slices, or cross-sections, corresponding to the values $z=0.3$, $z=0.45$ and $z=0.6$.  At each value for $z$ the gray shaded
area is the difference between $\conv(\Fcalprime)$ as given by the three SOC constraints from \samb{Proposition} $\ref{LEM:LBUB}$
and the region determined by the center SOC constraint \eqref{EQ:LBUBcenter} combined with the RLT constraints
\eqref{EQ:RLTx} -- \eqref{EQ:RLTy}.

\begin{figure}
    \centering
   \includegraphics[height=4in]{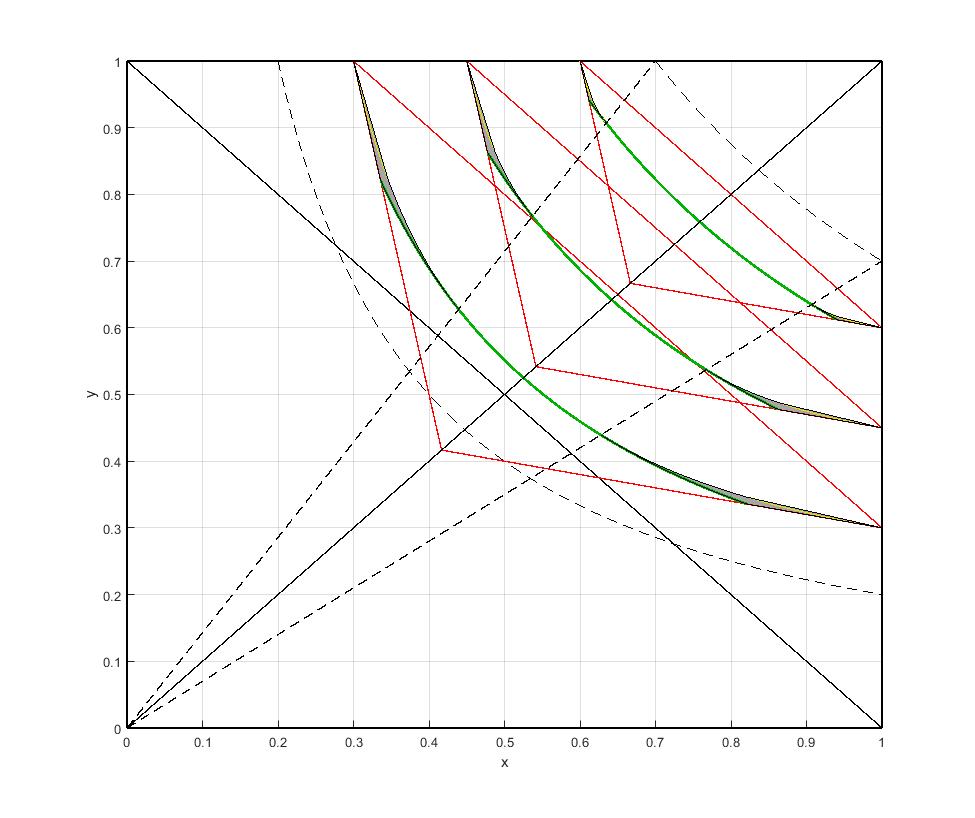}
   \caption{Cross-sections of convex hull vs. center cone only with RLT constraints}
   \label{FG:DBslicer}
\end{figure}

In addition to the approximation based on one SOC constraint, it is possible to give an exact disjunctive represention
of $\conv(\Fcalprime)$ over
the entire region corresponding to the bounds $l=(0,0,{\lz})$, $u=(1,1,{\uz})$ by using additional variables
$(\lam_i,x_i,y_i,z_i)$, $i=1,2,3$, where $\lam\ge 0$, ${\uz} x_1\le y_1 \le x_1/{\uz}$, $y_2\le {\uz} x_2$, $x_3\le {\uz} y_3$, and
\[
x=\sum_{i=1}^3 x_i, \quad y=\sum_{i=1}^3 y_i, \quad z=\sum_{i=1}^3 z_i,\quad\sum_{i=1}^3 \lam_i=1.
\]
Each $(x_i,y_i,z_i)$ is then constrained to be in one of the regions given in \samb{Proposition} \ref{LEM:LBUB}, homogenized using
the variable $\lam_i$. We omit the straightforward details.

%%% leave out details of disjunctive representation
\leaveout{
In the resulting constraints the bounds on $z_i$ are replaced by $\lam_i {\lz}\le z_i\le \lam_i {\uz}$, and the RLT constraint \eqref{EQ:RLTback}
 is replaced by the constraint $z_i\ge x_i+y_i-\lam_i$, $i=1,2,3$.
In addition, in the SOC constraint for $(x_1,y_1,z_1)$,
$\sqrt{{\lz}{\uz}}$ is replaced with $\lam_1\sqrt{{\lz}{\uz}}$;  in the SOC constraint for
$(x_2,y_2,z_2)$ the definitions of $(\xhat_2,\yhat_2)$ are changed to
$\xhat_2=(\lam_2-x_2)$ and $\yhat_2=(\lam_2{\uz}-y_2)$, and  in the SOC constraint for
$(x_3,y_3,z_3)$ the definitions of $(\xhat_3,\yhat_3)$ are changed to $\xhat_3=(\lam_3{\uz}-x_3)$, $\yhat_3=(\lam_3-y_3)$.}

\subsection{Volume computation}\label{SS:volume}

As an application of the above results, in this section we will compare the volumes of $\conv(\Fcalprime)$ that are obtained by
applying the SOC constraints described in \samb{Propositions} \ref{LEM:UB} and \ref{LEM:LB} to the volumes of the
regions corresponding to the RLT constraints and the simple bound constraints $z\le {\uz}$ or $z\ge {\lz}$ \samb{(but not both)}.
Computing these volumes will also allow us to compute the total volume reduction that is obtained by creating
two subproblems, one corresponding to impoing an upper bound $z\le b$ and the other a lower bound $z\ge b$.

In the case of an upper bound $z\le {\uz}$, it is straightforward to compute that the volume of the RLT region
with the additional constraint $z\le {\uz}$ is ${\uz}({\uz}^2-3{\uz}+3)/6$, and \samb{using a simple integration calculation}, the volume removed by adding the SOC
constraint in \samb{Proposition} \ref{LEM:UB} is ${\uz}^2({\uz}-1-\ln({\uz}))/3$.  The volume of $\conv(\Fcalprime)$ with
bounds $l=(0,0,0)$, $u=(1,1,{\uz})$ is therefore
\begin{equation}\label{EQ:UBvolume}
\frac{u}{6}\left(3+2{\uz}\ln({\uz}) - {\uz} -{\uz}^2\right).
\end{equation}

In the case of a lower bound $z\ge {\lz}$, the volume of the RLT region with the added constraint $z\ge {\lz}$ is
$(1+{\lz})^3/6$, where here we impose the RLT constraints \eqref{EQ:RLTx} -- \eqref{EQ:RLTy} using ${\lx}={\ly}=0$.
The volume removed by adding the SOC constraint in \samb{Proposition} \eqref{LEM:LB} can be computed to be
${\lz}(1-{\lz})({\lz}-1-\ln({\lz}))/3$.  The volume  of $\conv(\Fcalprime)$ with
bounds $l=(0,0,{\lz})$, $u=(1,1,1)$ is therefore
\begin{equation}\label{EQ:LBvolume}
\frac{1-{\lz}}{6}\left(1+2{\lz}\ln({\lz})-{\lz}^2\right).
\end{equation}

\begin{figure}
    \centering
   \includegraphics[width=\textwidth]{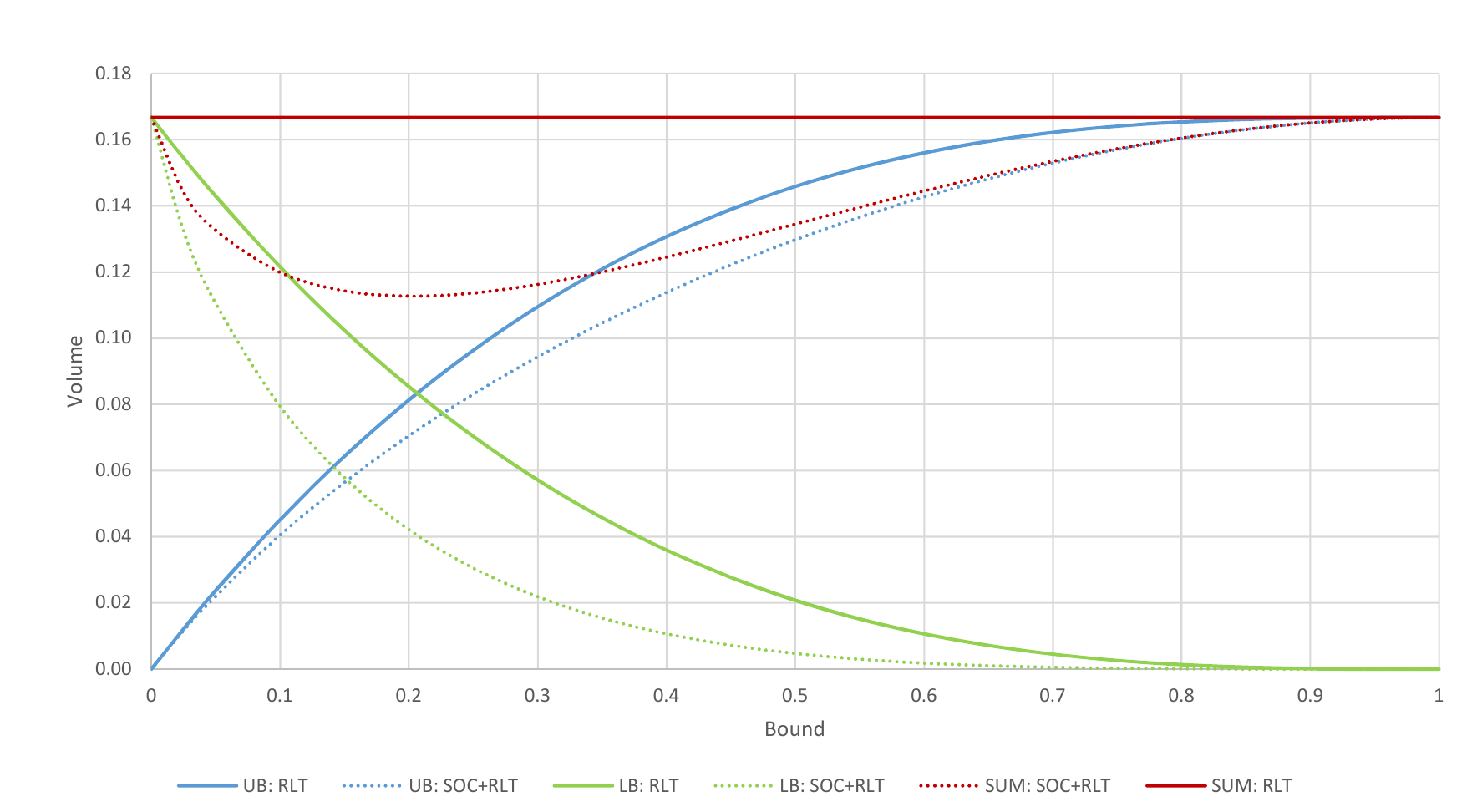}
   \caption{Volume comparisons for convex hulls versus RLT regions with added bounds on $z$.}
   \label{FG:VolComp}
\end{figure}

We illustrate these volume computations in Figure \ref{FG:VolComp}. \samb{Let
$b$ represent the bound depicted on the horizontal access.} In
the figure the UB: SOC+RLT series shows the volume of
$\conv(\Fcalprime)$ with an upper bound ${\uz} \samb{= b}$ from
\eqref{EQ:UBvolume}, and for comparison UB: RLT shows the
volume of the RLT region cut at $z={\uz} \samb{=b}$. The LB:
SOC+RLT series similarly shows the volume of $\conv(\Fcalprime)$
with a lower bound ${\lz \samb{=b}}$ from \eqref{EQ:UBvolume}, and for
comparison LB: RLT shows the volume of the RLT region cut at
$z={\lz} \samb{= b}$. The SUM: SOC+RLT series shows the sum
of the two volumes from \eqref{EQ:UBvolume} and \eqref{EQ:LBvolume}
if ${\lz}={\uz}=b$. The sum of the volumes of the two RLT regions,
one cut from below at ${\lz}=b$ and the other cut from above at
${\uz}=b$, is constant and equal to $1/6$. From the chart it is evident
that the sum of the volumes of the two convex hulls is minimized at
approximately $b=0.2$; the exact minimizer satisfies the nonlinear
equation $\ln(b)=2(b-1)$. In Figure \ref{FG:VolRat}, we graph the ratio
of the volume \eqref{EQ:UBvolume} to that of the RLT region cut at
${\uz}=b$, the ratio of the volume \eqref{EQ:LBvolume} to that of the
RLT region cut at ${\lz}=b$, and the ratio of the sum of the two volumes
to that of the total RLT region. The volume of the sum is reduced by
approximately 32.4\% at the minimizing value. This has an interesting
interpretation as the possible effect of applying spatial branching
to the continuous variable $z$, where one subproblem has an upper
bound ${\uz}=b$ and the other has a lower bound ${\lz}=b$. In Figure
\ref{FG:DI}, we illustrate the effect of such a branching by showing the
convex hulls for ${\uz}=0.3$ and ${\lz}=0.3$; in this case a total of
approximately 30\% of the volume of the original RLT region is removed
by considering the two subproblems. See \cite{Lee.Skipper.Speakman} for
a recent survey of volume-based comparisons of polyhedral relxations for
nonconvex optimization, and \cite{Speakman.Lee} for an application to
branching-point selection in the presence of trilinear terms.

\begin{figure}
    \centering
   \includegraphics[width=\textwidth]{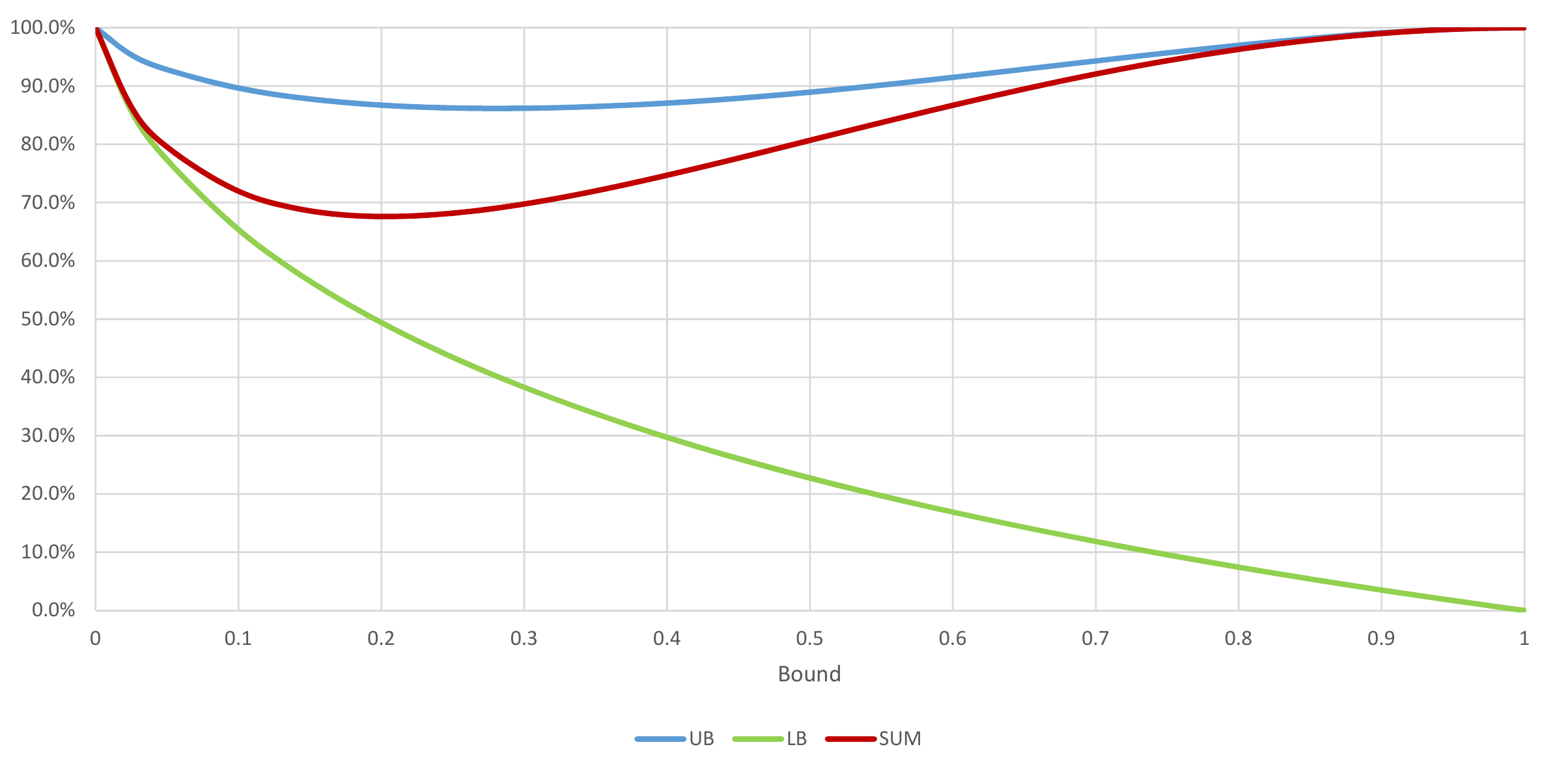}
   \caption{Volume ratios between convex hulls and RLT regions with added bounds on $z$}
\label{FG:VolRat}
\end{figure}

\begin{figure}
    \centering
   \includegraphics[height=3.5in]{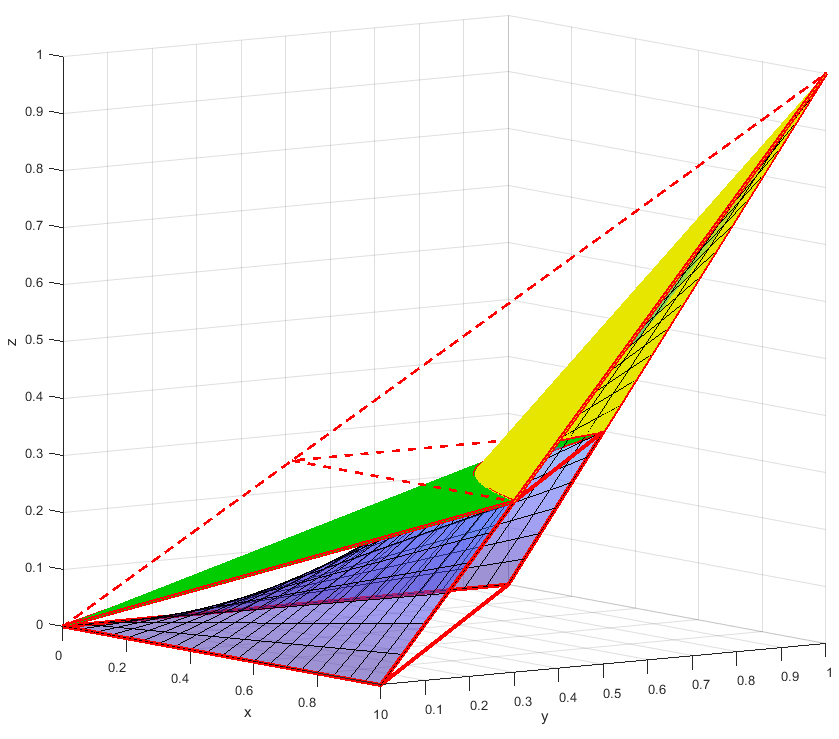}
   \caption{Effect of spatial branching on $z$}
   \label{FG:DI}
\end{figure}

\section{\samb{Convex hull representation with general $(\lx,\ly)$}}\label{SEC:Gen}

In this section, we consider the case where the original variables
$(x,y)$ have more general bounds of the form ${\lx}\leq x\leq {\ux}$,
${\ly}\leq y\leq {\uy}$. \samb{In particular, $\lx$ and $\ly$ can
be positive.} We continue to assume without loss of generality that
${\ux}={\uy}=1$ since this can always be achieved by a simple rescaling
of $x$ and/or $y$. Furthermore, as discussed in the introduction, we now
assume without loss of generality that \eqref{eq:bound_assumptions} holds.

% Note that if $z=xy$ and ${\ly}>0$ then $x\le {\uz}/{\ly}$, so we
% could assume that
% ${\ux}\le {\uz}/{\ly}$. Then ${\ux}=1$ implies that we can assume ${\ly}\le {\uz}$, and similarly ${\lx}\le {\uz}$.  In addition
% $x\ge {\lz}/{\uy}$, so ${\uy}=1$ implies that we may assume that ${\lx}\ge {\lz}$ and similarly ${\ly}\ge {\lz}$. Combining these
% facts we assume throughout that
% \[
% {\lz}\le {\lx}\le {\uz},\quad {\lz}\le {\ly}\le {\uz}.
% \]

\subsection{\samb{Non-trivial} lower bound on $xy$ with general $(\lx,\ly)$} \label{SS:LBgen}

With general lower bounds on $(x,y)$ and a \samb{non-trivial} lower bound on the product $z$,
$\conv(\Fcalprime)$ can be described almost identically to the representation given in \samb{Proposition} \ref{LEM:LB}
for the case of ${\lx}={\ly}=0$.

\begin{lemma}\label{LEM:LBgen}
Let $l=({\lx},{\ly},{\lz})$, $u=(1,1,1)$ where $0 \le \lx \ly < \lz < 1$. Then $\conv(\Fcalprime)$ is given by the RLT constraints \eqref{EQ:RLT}, the bound $z\ge {\lz}$ and the SOC constraint $\sqrt{(\xhat,\yhat) M (\xhat,\yhat)\tran}
\le x+y-2z$ where $\xhat\coloneqq1-x$, $\yhat\coloneqq1-y$, and
\[
    M=\begin{pmatrix}1 & 2l\samb{_z}-1\\2l\samb{_z}-1 & 1\end{pmatrix}\gesem 0.
\]
\end{lemma}

\begin{proof}
The construction of the SOC constraint that implies the lifted tangent inequalities is identical to the case of
${\lx}={\ly}=0$ considered in the proof of \samb{Proposition} \ref{LEM:LB}, and this SOC constraint together with the RLT
constraints \eqref{EQ:RLT} and the bound $z\le {\uz}$ gives \samb{$\conv(\Fcalprime)$}.  However, \samb{in contrast to Proposition
\ref{LEM:LB}}, if ${\lx}>{\lz}$ then
the constraint \eqref{EQ:RLTx} is no longer redundant, if ${\ly}>{\lz}$ the constraint \eqref{EQ:RLTy} is no longer redundant, and in \samb{both} case\samb{s} the constraint \eqref{EQ:RLTbottom} is no longer redundant.
\end{proof}

In Figure \ref{FG:GenLBwLBc2}, we illustrate $\conv(\Fcalprime)$ \samb{for} $l=(0.5,0.3,0.3)$, $u=(1,1,1)$. Since
${\lx}>{\lz}$, the constraints \eqref{EQ:RLTx} and \eqref{EQ:RLTbottom} are now active.

\begin{figure}[h!]
    \centering
   \includegraphics[height=3.5in]{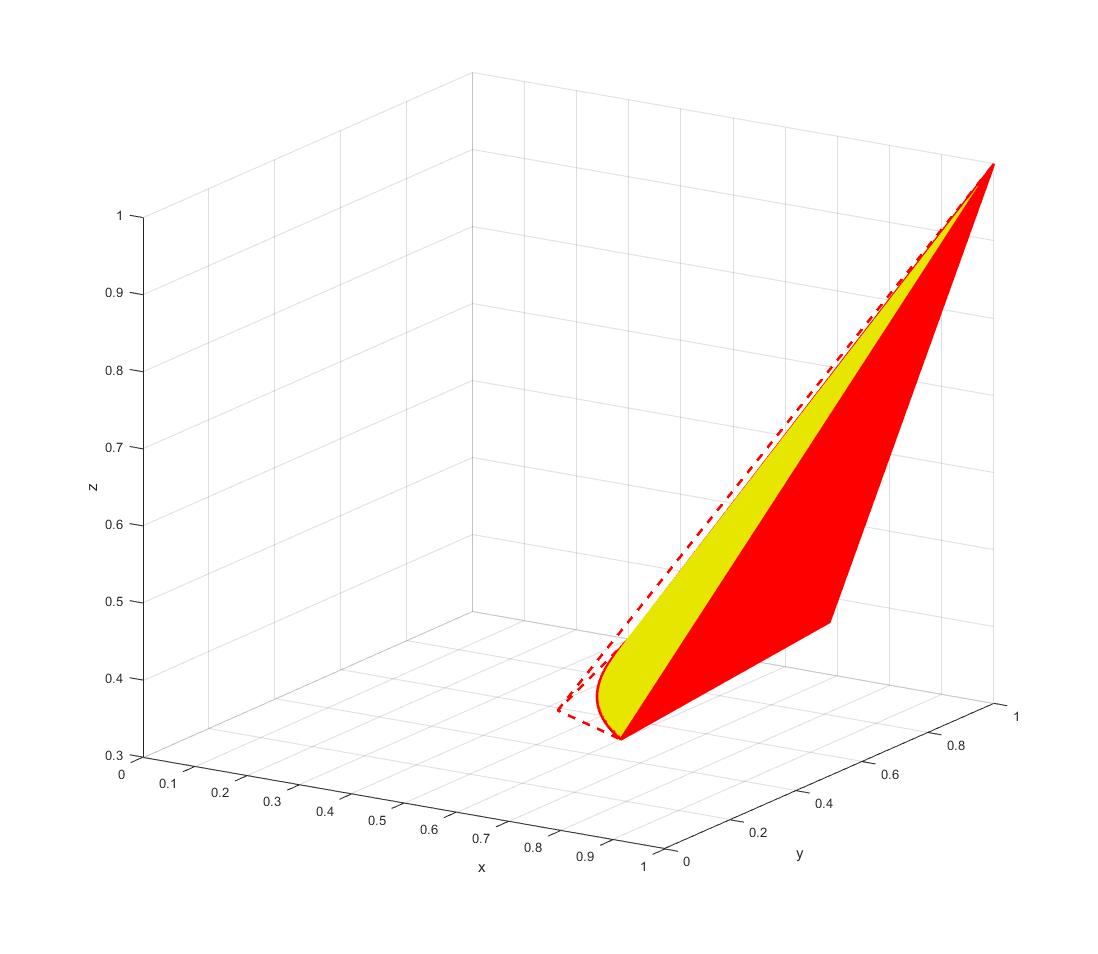}
   \caption{Convex hull with general lower bounds on $(x,y,z)$.}
   \label{FG:GenLBwLBc2}
\end{figure}

\subsection{\samb{Non-trivial} upper bound on $xy$ with general $(\lx,\ly)$} \label{SS:UBgen}

With general lower bounds $\lx$ and $\ly$ \samb{and a non-trivial upper bound on $z$}, the geometry of $\conv(\Fcalprime)$ is similar to the case of
$\lx=\ly=0$ considered in Section \ref{SS:UB}, but the derivation of \samb{the} conic constraint in SOC form is
 more complex.  Lifted tangent inequalities now correspond to line segments joining a point
 $(\xbar,\ybar,u_z) \samb{\in \Fcalprime}$ with the point $(\lx,\ly,\lx\ly)$.  For a point $(x,y,z)$ on such a line segment
we have $x=\alpha\lx+(1-\alpha)\xbar$, $y=\alpha\ly+(1-\alpha)\ybar$.  Writing $(\xbar,\ybar)$ in terms of $(x,y)$
then results in a quadratic equation for $\alpha$, and for the appropriate root of this equation the constraint
$x\le \alpha(\lx\ly)+(1-\alpha)\uz$ results in the constraint
\begin{equation}\label{EQ:UBgen}
(z-\ly x)(z-\lx y) \le \uz(x-\lx)(y-\ly).
\end{equation}
This constraint is certainly valid for all $(x,y,z)\in \conv(\Fcalprime)$.  In particular, if $z=xy\le\uz$ then
$(z-\ly x)= x(y-\ly)$ and $(z-\lx y)= y(x-\lx)$, so $(z-\ly x)(z-\lx y) = xy (x-\lx)(y-\ly)\le \uz(x-\lx)(y-\ly)$.
Note that if $\lx=\ly=0$, then \eqref{EQ:UBgen} is exactly the SOC constraint $z^2\le \uz xy$ from \samb{Proposition} \ref{LEM:UB}.
If either $\lx=0$ or $\ly=0$ it is also easy to put the constraint \eqref{EQ:UBgen} into the form of an SOC constraint,
but when $\lx>0$, $\ly>0$ this is nontrivial.

\begin{lemma}\label{LEM:UBgen}
Let $l=({\lx},{\ly},0)$, $u=(1,1,\uz)$ where $0<{\uz}<1$. Then $\conv(\Fcalprime)$ is given by the RLT constraints \eqref{EQ:RLT}, the bound $z\le {\uz}$ and the SOC constraint
\begin{equation}\label{EQ:UBgenSOC}
\uz(z-\lx\ly)^2\le \big(\uz(x-\lx)+\lx(z-\ly x)\big)\big(\uz(y-\ly)+\ly(z-\lx y)\big).
\end{equation}
\end{lemma}

\begin{proof}
The convex hull of $\Fcalprime$ is given by the RLT constraints, the bound
$z\le \uz$, and the lifted tangent inequalities, and the latter are implied by the constraint \eqref{EQ:UBgen}.
By a direct computation the constraint \eqref{EQ:UBgenSOC} is equivalent to multiplying both sides of \eqref{EQ:UBgen}
by the constant $\uz-\lx\ly>0$.  Moreover, $x\ge \lx$, $y\ge \ly$ and the RLT constraint
\eqref{EQ:RLTbottom} together imply that $z \ge \ly x$ and $z\ge \lx y$. Both terms that form the product on the right-hand side of \eqref{EQ:UBgenSOC} can therefore be assumed to be nonnegative, so \eqref{EQ:UBgenSOC} is an
SOC constraint that implies the lifted tangent inequalities.
\end{proof}

The proof of \samb{Proposition} \ref{LEM:UBgen} requires only that \eqref{EQ:UBgen} and \eqref{EQ:UBgenSOC} are equivalent,
but it is worth noting how \eqref{EQ:UBgenSOC} was obtained.  This was accomplished by writing \eqref{EQ:UBgen}
in the form $v\tran Q v\le 0$, where $v=(1,x,y,z)\tran$, and then performing symbolic, symmetric transformations on $Q$
so as to obtain
\[
SQS\tran =
\Qhat=\begin{pmatrix} 2\uz & 0 & 0 &0\\  0 & 0 & -1 & 0 \\ 0 & -1 & 0 & 0 \\  0 & 0 & 0 &0\end{pmatrix}.
\]
Note that $v\tran Qv=v\tran S^{-1} \Qhat S^{-T}v$, and $\Qhat$ has exactly one negative eigenvalue.  The spectral decomposition of $\Qhat$ and the symbolic matrix $S^{-T}$ were together used to obtain the equivalent SOC constraint
\eqref{EQ:UBgenSOC}. In Figure \ref{FG:GenLBwUB}, we illustrate $\conv(\Fcalprime)$ for the case with
$u_z=0.7$ and lower bounds $l_x=0.4$, $l_y=0.5$.

\begin{figure}[h!]
    \centering
   \includegraphics[height=3.5in]{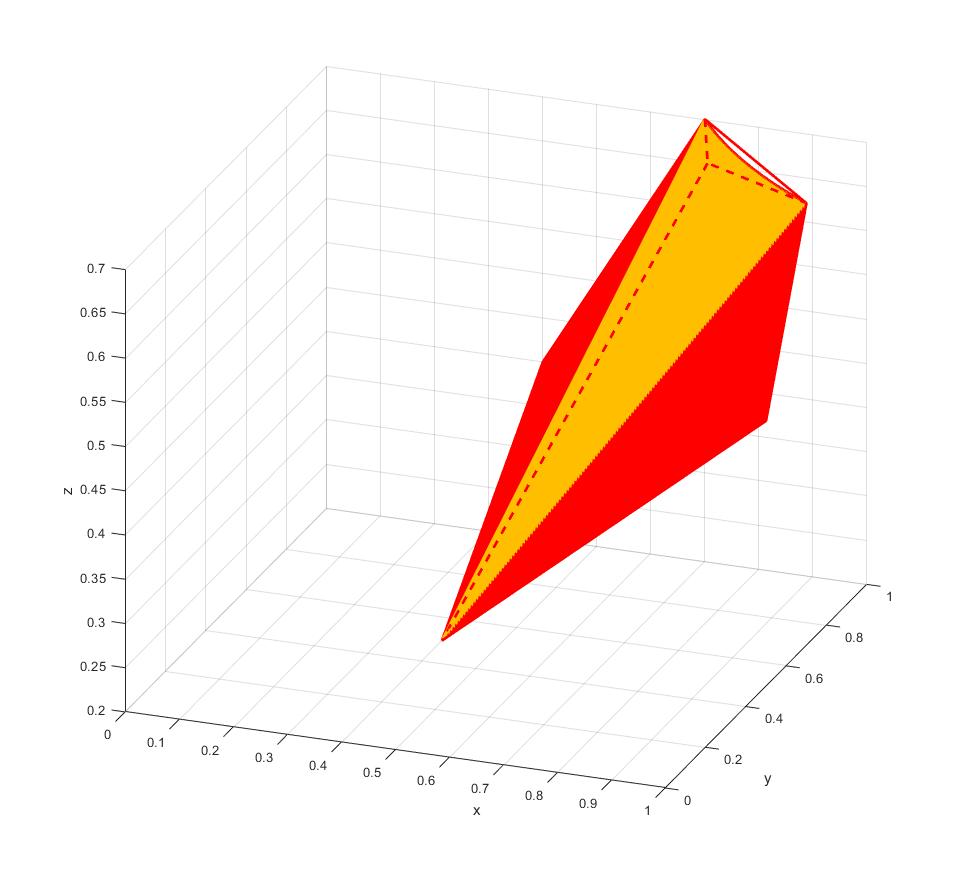}
   \caption{Convex hull with general lower bounds on $(x,y)$ and \samb{non-trivial} upper bound on $z$}
   \label{FG:GenLBwUB}
\end{figure}

\subsection{\samb{Non-trivial} lower and upper bounds on $xy$ with general $(\lx,\ly)$}\label{SS:LBUBgen}

We now consider the most general case for $\Fcalprime$, where $l=(\lx,\ly,\lz)>0$ and $\uz<1$.  We continue
to assume that $\ux=\uy=1$, and $\lz\le {\lx}\le {\uz}$, ${\lz}\le {\ly}\le {\uz}$ as described at the beginning of
the section.  Finally we assume that $\lx\ly<\lz$, since \samb{otherwise} $\lx\ly\ge\lz$ implies that the lower bound $xy\ge \lz$ is
redundant, \samb{which is} the case of the previous section.

In order to describe the possible representations for $\conv(\Fcalprime)$, it is very convenient to dissect the
domain for possible values of $(\lx,\ly)$ into regions where representations of a particular type occur.  These
regions naturally involve the values $\sqrt{\lz\uz}$ and $\sqrt{\lz/\uz}$. In particular, note that the point
$(x,y)=(\sqrt{\lz/\uz},\sqrt{\lz\uz})$ is the intersection of the line $y=\uz x$ and the curve $xy=\lz$, while
$(x,y)=(\sqrt{\lz\uz},\sqrt{\lz/\uz})$ is the intersection of the line $x=\uz y$ and the curve $xy=\lz$. Under
our assumptions for the values of $l$ and $u$, the possible regions for $(\lx,\ly)$ are as follows and are illustrated
in Figure \ref{FG:domains} for the case of $\lz=0.1$, $\uz=0.7$.
\begin{itemize}\baselineskip16pt
\item[A.] $\lx\ge \sqrt{\lz\uz}$, $\ly\ge \sqrt{\lz\uz}$, $\lx\ly<\lz$.
\item[B.] $\lz\le\lx\le\sqrt{\lz\uz}$, $\lz\le\ly\le \sqrt{\lz\uz}$.
\item[C.] $\lz\le\lx\le\sqrt{\lz\uz}$, $\sqrt{\lz\uz}\le\ly\le \sqrt{\lz/\uz}$.
\item[D.] $\lx\ge\lz$, $\sqrt{\lz/\uz}\le \ly\le\uz$, $\lx\ly\le\lz$.
\item[E.] $\sqrt{\lz\uz}\le\lx\le \sqrt{\lz/\uz}$, $\lz\le\ly\le\sqrt{\lz\uz}$.
\item[F.] $\sqrt{\lz/\uz}\le \lx\le\uz$, $\ly\ge\lz$, $\lx\ly\le\lz$.
\end{itemize}

\begin{figure}[h!]
    \centering
   \includegraphics[height=3.5in]{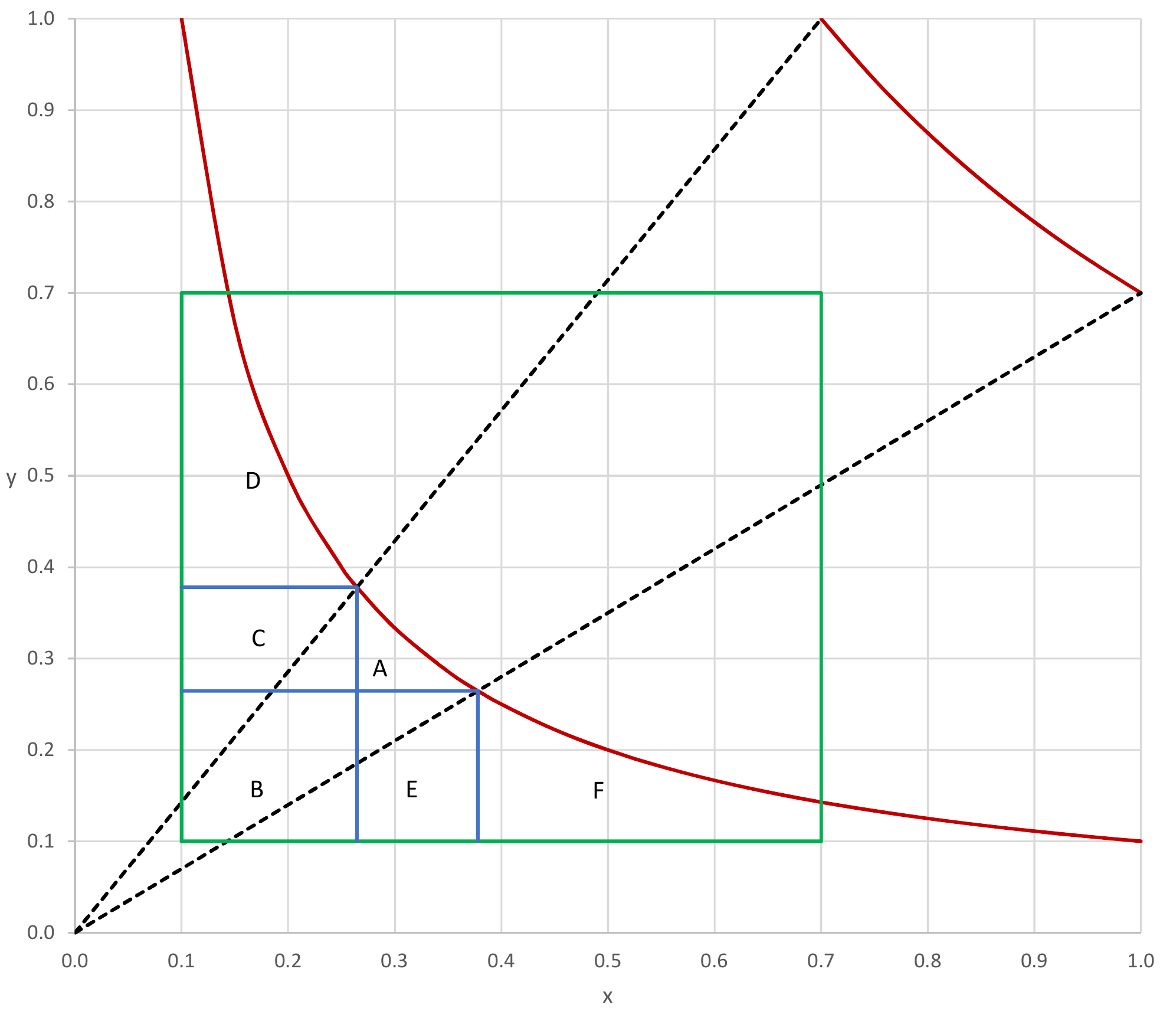}
   \caption{Domains for $({\lx},{\ly})$ with $\lz=0.1$, $\uz=0.7$}
   \label{FG:domains}
\end{figure}

It is clear that regions E and F correspond to regions C and D, respectively, with the roles of $x$ and $y$ interchanged.
Since we can assume without loss of generality that $\lx\le \ly$, in the results below we will only
consider regions A--D.  We omit proofs of these results since in all cases they are based on SOC representations
for lifted tangent inequalities described in earlier sections.  In each of the four cases, the representation of
 $\conv(\Fcalprime)$ will include several SOC constraints that imply the lifted tangent inequalities on different $(x,y)$ domains.  In Figure \ref{FG:ABCDcones}, we illustrate these domains using values of $(\lx,\ly)$ corresponding to each of the regions A--D, with $\lz=0.1$, $\uz=0.7$ as in Figure  \ref{FG:domains}.  In the figure, the boundaries of
 domains on which different SOC constraints imply the lifted tangent inequalities are given by solid black lines, and
 blue lines indicate the region $(x,y)\ge (\lx,\ly)$.

\begin{figure}[h!]
    \centering
   \includegraphics[height=5in]{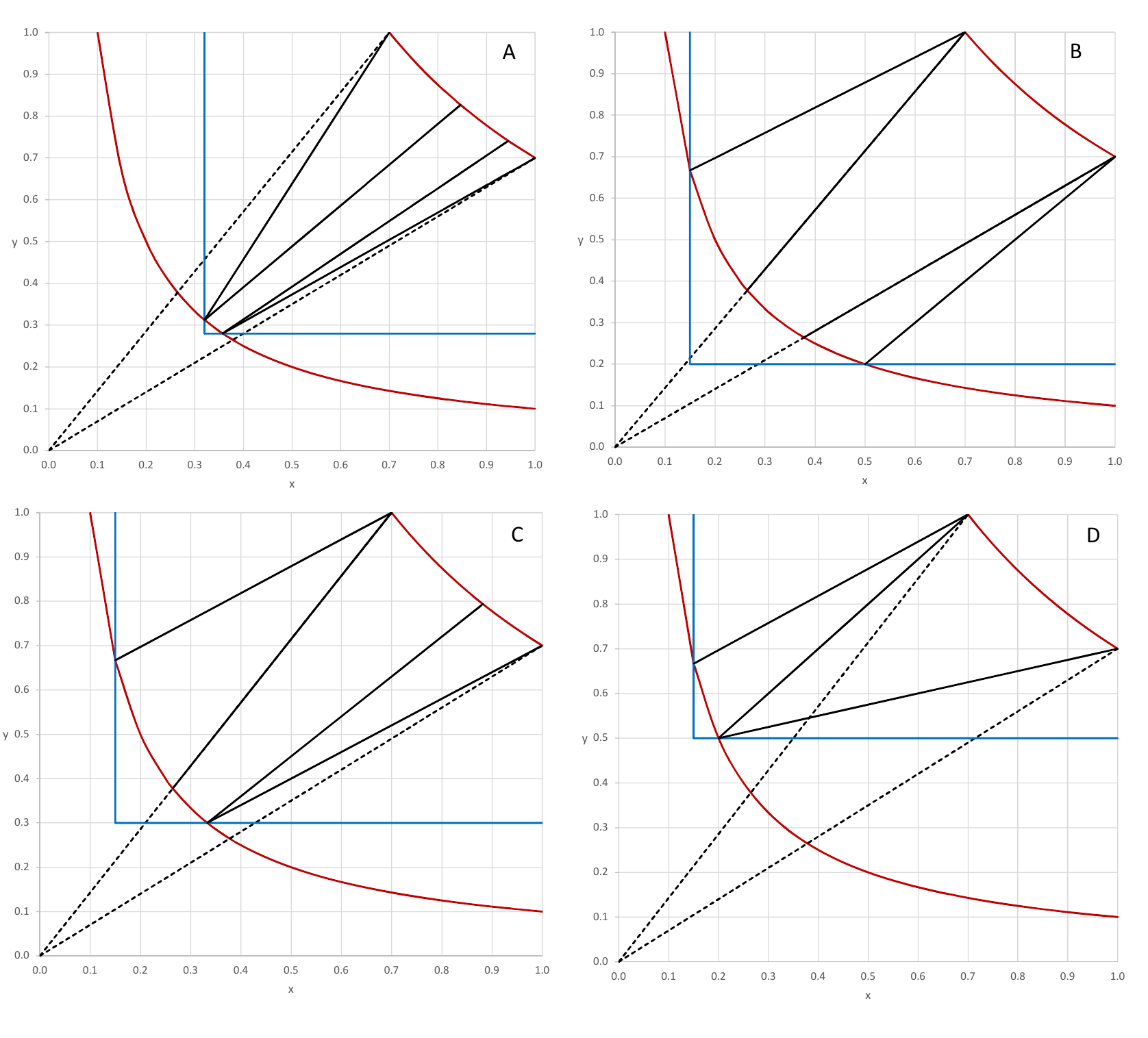}
   \caption{Domains for SOC constraints with $\lz=0.1$, $\uz=0.7$}
   \label{FG:ABCDcones}
\end{figure}

\begin{lemma}\label{LEM:regionA}
Suppose that $(\lx,\ly)$ is in region A. Then $\conv(\Fcalprime)$ is given by the the RLT constraints, the bounds
$\lz\le z\le \uz$, and the following three SOC constraints, each applicable in a different region:
\begin{enumerate}\baselineskip16pt
\item The constraint \eqref{EQ:LBUBcenter}, applicable if
$y\ge (\ly^2/\lz)x$, $y\le (\lz/\lx^2)x$.
\item The constraint \eqref{EQ:UBgenSOC}, but with $\lx$ replaced by $\lz/\ly$, applicable if
$y\le (\ly^2/\lz)x$.
%This constraint is valid
%or $x\ge \lz/\ly$ and is binding in the region  $y\le (\ly^2/\lz)x$, $y\ge \frac{1}{\ly-\lz}(\ly(\uz-\ly)x+\ly^2-\lz\uz)$.
\item The constraint \eqref{EQ:UBgenSOC}, but with $\ly$ replaced by $\lz/\lx$, applicable if
$y\ge (\lz/\lx^2)x$.
%for $y\ge \lz/\lx$ and is binding in the region $x\le (\lx^2/\lz)y$, $x\ge \frac{1}{\lx-\lz}(\lx(\uz-\lx)y+\lx^2-\lz\uz)$.
\end{enumerate}
\end{lemma}

Note that the first constraint in \samb{Proposition} \ref{LEM:regionA} is exactly the constraint based on $(l_z,u_z)$ from
\samb{Proposition} \ref{LEM:LBUB}. This constraint is globally valid and is binding in the region $y\ge (\ly^2/\lz)x$, $y\le (\lz/\lx^2)x$.
The second constraint corresponds to using the lower bounds
$(\lz/\ly,\ly)$ in \samb{Proposition} \ref{LEM:UBgen}, and is certainly then valid for all $(x,y)\ge (\lz/\ly,\ly)$, where $\lz/\ly>\lx$ by
assumption. Note also that $y\le (\ly^2/\lz) x$ and $ y\ge \ly$ together imply that $x\ge \lz/\ly$. Similarly the third constraint is valid for all  $(x,y)\ge (\lx,\lz/\lx)$.  The regions on which the second
and third constraints are actually binding can easily be determined from the points $(\lz/\ly,\ly)$, $(\lx,\lz/\lx)$, $(1,\uz)$ and $(\uz,1)$;
see Figure \ref{FG:ABCDcones}.

For $(\lx,\ly)$ in region B, the representation with lower bounds $(\lx,\ly)>0$ is essentially identical to that
given in \samb{Proposition} \ref{LEM:LBUB}, except that the RLT constraints can now all be active.

\begin{lemma}\label{LEM:regionB}
Suppose that $(\lx,\ly)$ is in region B. Then $\conv(\Fcalprime)$ is given by the RLT constraints, the bounds
$\lz\le z\le \uz$, and the three SOC constraints from \samb{Proposition} \ref{LEM:LBUB}, where each constraint is
applicable for the $(x,y)$ values as given in \samb{Proposition} \ref{LEM:LBUB}.
\end{lemma}

For $(\lx,\ly)$ in region C, the representation with lower bounds $(\lx,\ly)>0$ uses a mixture of the SOC constraints that appear
in \samb{Propositions} \ref{LEM:regionA} and \ref{LEM:regionB}.

\begin{lemma}\label{LEM:regionC}
Suppose that $(\lx,\ly)$ is in region C. Then $\conv(\Fcalprime)$ is given by the RLT constraints, the bounds
$\lz\le z\le \uz$, and the following three SOC constraints,  each applicable on a different region:
\begin{enumerate}\baselineskip16pt
\item The constraint $(z+\sqrt{{\lz}{\uz}})^2\le (\sqrt{{\lz}}+\sqrt{{\uz}})^2xy$, applicable if
$y\ge (\ly^2/\lz)x$, $y\le x/\uz$.
\item The constraint \eqref{EQ:UBgenSOC}, but with $\lx$ replaced by $\lz/\ly$, applicable if
$y\le (\ly^2/\lz)x$.
\item The constraint $\sqrt{(\xhat,\yhat)M_2(\xhat,\yhat)\tran}\le
x+{\uz}y-2z$, where $\xhat\coloneqq {\uz}-x$, $\yhat\coloneqq1-y$ and $M_2\gesem 0$ is given in \eqref{EQ:M1M2}, applicable if $x\le {\uz} y$.
\end{enumerate}
\end{lemma}

\begin{figure}
\begin{subfigure}{.5\textwidth}
  \centering
  % include first image
  \includegraphics[width=.9\linewidth]{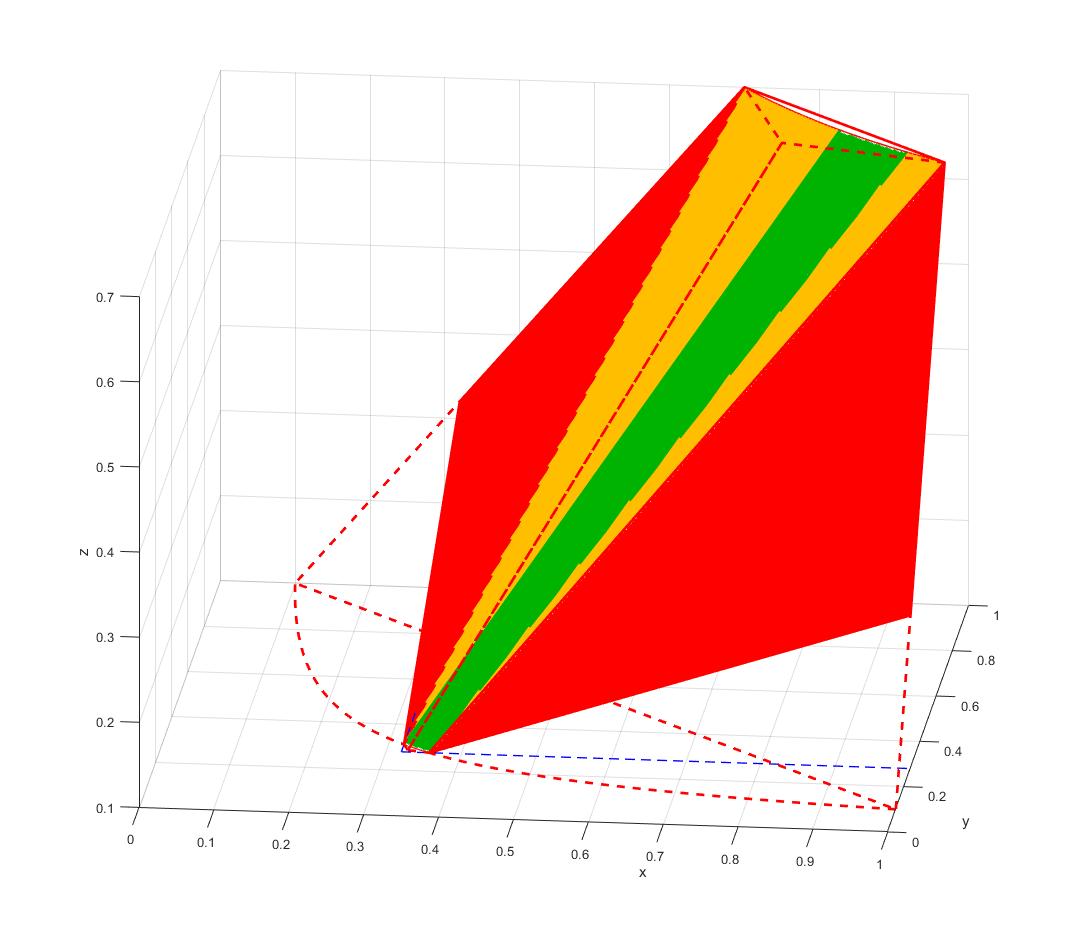}
  \caption{Case A: $l_x=0.32$, $l_y=0.28$}
  \label{fig:3.3A}
\end{subfigure}
\begin{subfigure}{.5\textwidth}
  \centering
  % include second image
  \includegraphics[width=.9\linewidth]{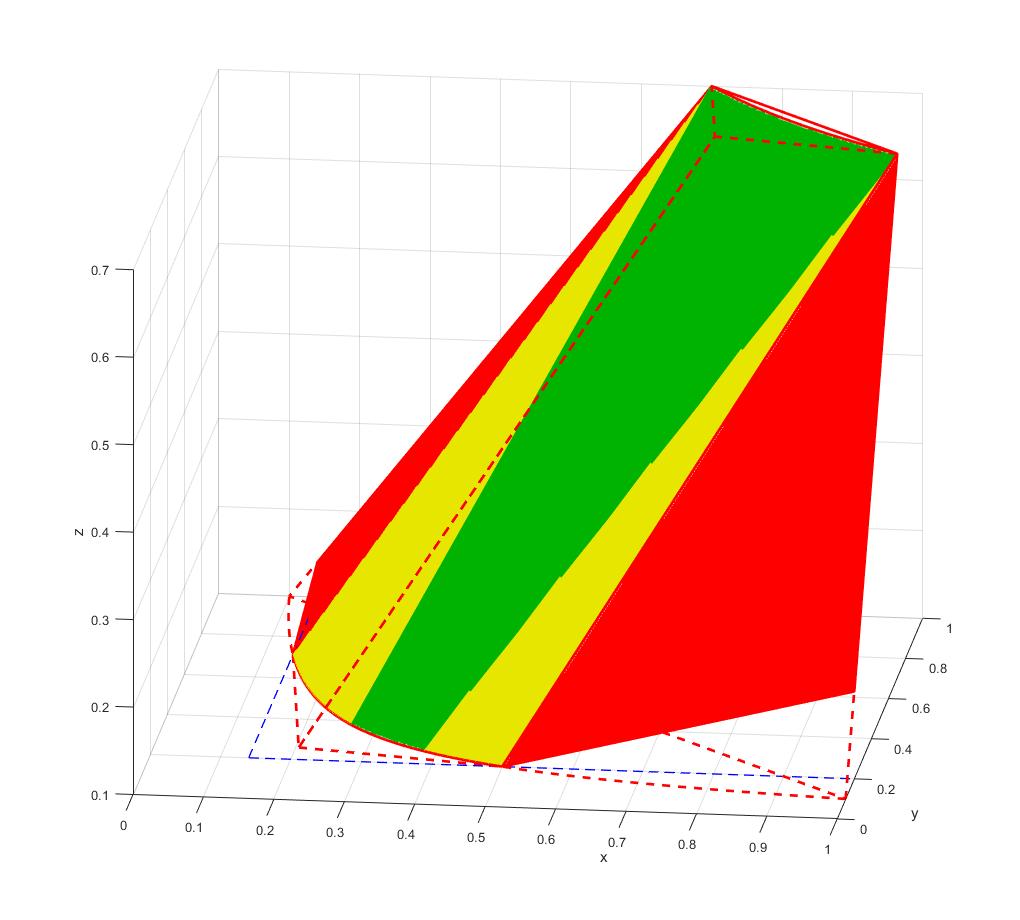}
  \caption{Case B: $l_x=0.14$, $l_y=0.2$}
  \label{fig:3.3B}
\end{subfigure}

\medskip

\begin{subfigure}{.5\textwidth}
  \centering
  % include third image
  \includegraphics[width=.9\linewidth]{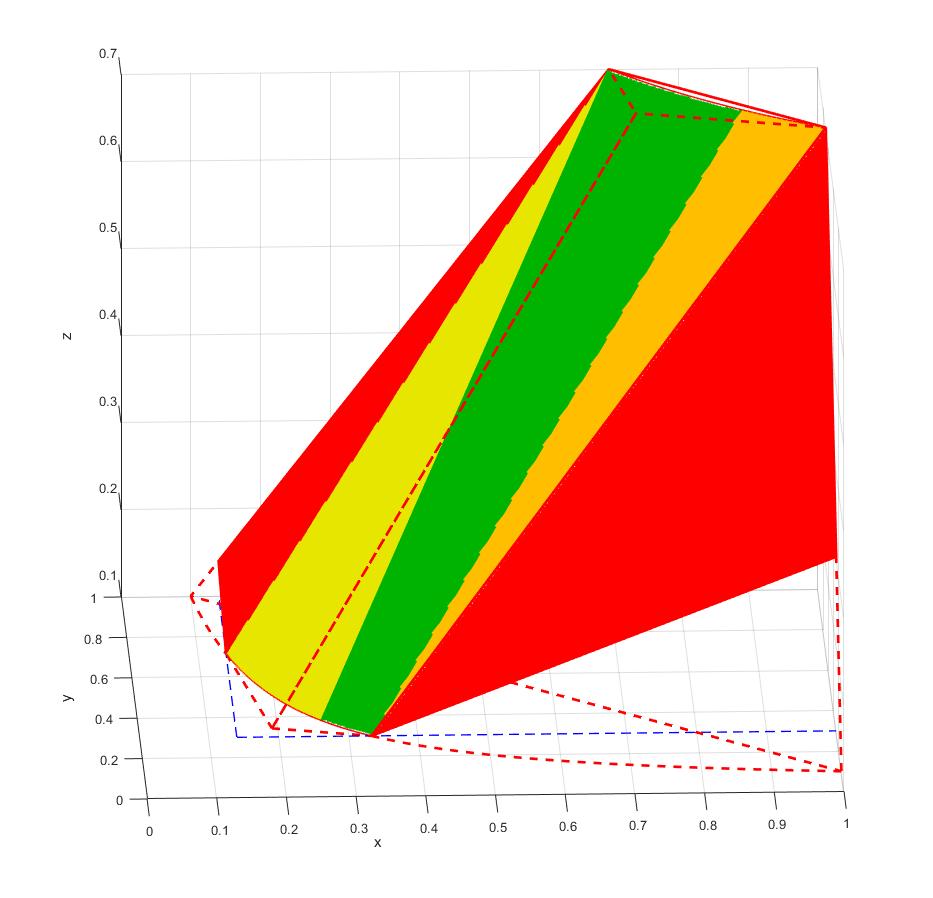}
  \caption{Case C: $l_x=0.14$, $l_y=0.3$}
  \label{fig:3.3C}
\end{subfigure}
\begin{subfigure}{.5\textwidth}
  \centering
  % include fourth image
  \includegraphics[width=.9\linewidth]{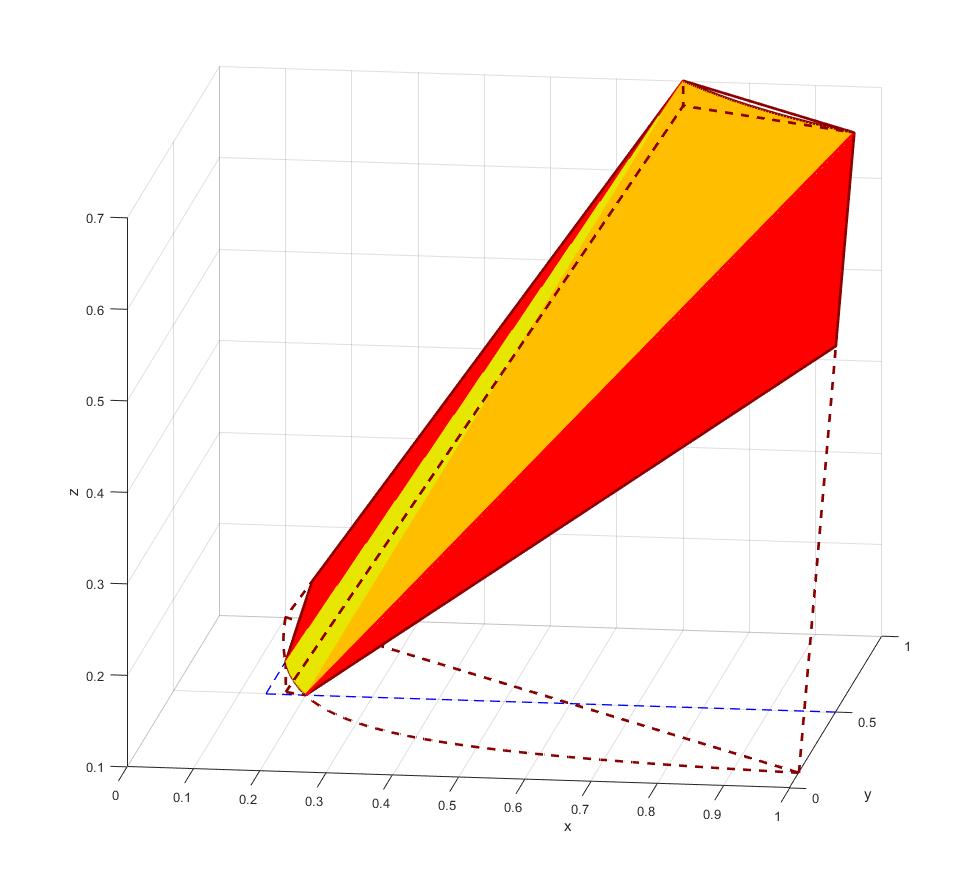}
  \caption{Case D: $l_x=0.14$, $l_y=0.5$}
  \label{fig:3.3D}
\end{subfigure}
\caption{Convex hulls with general lower bounds on $x$ and $y$: $l_z=0.1$, $u_z=0.7$}
\label{fig:3.3}
\end{figure}

The remaining case, where $(\lx,\ly)$ is in region D, is qualitatively different from the three previous cases because
the ``center cone'' from \samb{Proposition} \ref{LEM:LBUB} does not appear in the representation of $\conv(\Fcalprime)$.
There are only two SOC cones in the representation, and the boundary between the regions on which these cones are active is not homogeneous.  This boundary is given by the line which joins the points $(\lz/\ly,\ly)$ and $(\uz,1)$,
whose equation is $y=a+bx$, where
\begin{equation}\label{EQ:Dline}
a=\frac{\uz\ly^2-\lz}{\uz\ly-\lz},\quad b=\frac{\ly(1-\ly)}{\uz\ly-\lz}.
\end{equation}

\begin{lemma}\label{LEM:regionD}
Suppose that $(\lx,\ly)$ is in region D. Then $\conv(\Fcalprime)$ is given by the the RLT constraints, the bounds
$\lz\le z\le \uz$, and the following two SOC constraints,  each applicable in a different region:
\begin{enumerate}
\item The constraint \eqref{EQ:UBgenSOC}, but with $\lx$ replaced by $\lz/\ly$, applicable in the region
$y\le a + bx$, where $a$ and $b$ are given by \eqref{EQ:Dline}.
\item The constraint $\sqrt{(\xhat,\yhat)M_2(\xhat,\yhat)\tran}\le
x+{\uz}y-2z$, where $\xhat\coloneqq ({\uz}-x)$, $\yhat\coloneqq(1-y)$ and $M_2\gesem 0$ is given in \eqref{EQ:M1M2}, applicable if $y\ge a + bx$, where $a$ and $b$ are given by \eqref{EQ:Dline}.
\end{enumerate}
\end{lemma}

In Figure \ref{fig:3.3}, we illustrate examples of $\conv(\Fcalprime)$ for $l_z=0.1$, $u_z=0.7$ and
values of $l_x,l_y$ in each of the four regions A--D shown in Figure  \ref{FG:domains}.

Note that neither of the constraints in \samb{Proposition} \ref{LEM:regionD} is globally valid; the first is valid for $x\ge \lz/\ly$ and the second is valid for $y\ge x/\uz$. In fact,
all of the representations in this section involve some SOC constraints that are not globally valid. In order to represent
$\conv(\Fcalprime)$ over the entire set of feasible $(x,y)$ in any of these cases, one could use a disjunctive representation as
described at the end of Section \ref{SS:LBUB}.  Alternatively, we could always use the SOC constraint \eqref{EQ:LBUBcenter} together with the constraint \eqref{EQ:UBgenSOC} since both of these are globally valid.
 We would expect that these two SOC constraints together with the RLT constraints and bounds on $z$ would give a close approximation of  $\conv(\Fcalprime)$ in many cases.

 \section{Conclusion}\label{SEC:Conclusion}

 We have shown that in all cases, $\conv(\Fcalprime)$ can be represented using a combination of RLT constraints,  bound(s)
 on the product variable $z$ and no more than three SOC constraints.  In cases where more than one SOC constraint is required
 to represent $\conv(\Fcalprime)$, each such constraint is applicable on a subset of the domain of $(x,y)$ values, but
 one or two globally valid SOC constraints can be used together with the RLT constraints and bounds on $z$  to approximate
 $\conv(\Fcalprime)$.

Our results suggest a number of promising directions for future reserach.  First, it may be possible to extend some of these results to the
 case of multilinear terms, where $z$ is the product of $n>2$ variables; an extension of the
 lifted tangent inequalities to $n>2$ is described in \cite{belotti}. Second, it would be interesting to extend the convex hull description for the complete 5-variable
 system in \cite{AB} to allow for bounds on the product $xy$.  Note that the the results of \cite{AB} already apply to arbitrary bounds
 $0\le \lx\le x\le \ux$, $0\le \ly\le y\le \uy$, and bounds on the squared terms $x^2$ and/or $y^2$ are equivalent to bounds on the
 original variables $(x,y)$.  However, the results of \cite{AB} do not allow for an additional bound on the product $xy$. Finally, an extension of the results here to the case of $z=x\tran y$, where $x\in\Rbb_+^n$, $y\in\Rbb_+^n$, would be very significant
 since such bilinear terms appear in many applications.

 \section*{Acknowledgement}

Kurt Anstreicher would like to thank Pietro Belotti, Jeff Linderoth and Mohit Tawarmalani for very helpful conversations on the topic of this paper.

\bibliographystyle{abbrv}
\bibliography{xyz}

\end{document}